\documentclass[12pt, reqno]{amsart}

\usepackage{amsmath, amscd, amssymb, latexsym, amsthm, xspace, color}
\usepackage{ifpdf}
\usepackage{graphicx}
\usepackage[all]{xy}
\ifpdf
\else
\usepackage{epsfig}
\fi

\newcommand{\CC}{\mathbb{C}}

\newcommand{\NN}{\mathbb{N}}

\newcommand{\HH}{\mathbb{H}}

\newcommand{\PP}{\mathbb{P}}
\newcommand{\QQ}{\mathbb{Q}}
\newcommand{\RR}{\mathbb{R}}
\newcommand{\VV}{\mathbb{V}}
\newcommand{\ZZ}{\mathbb{Z}}

\newcommand{\cA}{{\mathcal A}}

\newcommand{\cE}{{\mathcal E}}
\newcommand{\cF}{{\mathcal F}}

\newcommand{\cM}{{\mathcal M}}

\newcommand{\cX}{{\mathcal X}}
\newcommand{\cY}{{\mathcal Y}}

\newcommand{\SL}{\mathop{SL}}

\newcommand{\SpgZ}{{\mathop{Sp}}_{2g}(\ZZ)}
\newcommand{\SpgQ}{{\mathop{Sp}}_{2g}(\QQ)}
\newcommand{\SpgR}{{\mathop{Sp}}_{2g}(\RR)}

\newcommand{\nul}{{\rm null}}
\newcommand{\even}{{\rm even}}

\DeclareMathOperator{\Sp}{Sp}

\newcommand{\GL}{{\rm GL}}

\renewcommand{\Re}{{\rm Re\,}}
\renewcommand{\Im}{{\rm Im\,}}

\newcommand{\Hyp}{{\rm Hyp}}

\newcommand{\eps}{\varepsilon}

\newcommand{\cq}{\tilde{q}}
\newcommand{\tS}{\tilde{S}}

\newcommand{\be}{\beta}
\newcommand{\de}{\delta}
\newcommand\e{\varepsilon}

\newcommand\tc[2]{\theta\Bigl[\begin{matrix}#1\\ #2\end{matrix}\Bigr]}
\newcommand\tch[2]{\Bigl[\begin{matrix}#1\\ #2\end{matrix}\Bigr]}
\newcommand\thetalc[2]{{\theta_{[{#1},{#2}]}}}
\newcommand\thetat[2]{{\theta\tch{#1}{#2}}}

\newcommand{\moduli}[1][g]{{\cM}_{#1}}
\newcommand{\AVmoduli}[1][g]{{\cA}_{#1}}
\newcommand{\Agdec}[1][g]{{\cA}_{#1}^{\rm dec}}
\newcommand{\barAVmoduli}[1][g]{{\overline{\cA}_{#1}}}

\def\be{\begin{equation}}   \def\ee{\end{equation}}     \def\bes{\begin{equation*}}    \def\ees{\end{equation*}}
\def\ba{\be\begin{aligned}} \def\ea{\end{aligned}\ee}   \def\bas{\bes\begin{aligned}}  \def\eas{\end{aligned}\ees}

\theoremstyle{definition}
\newtheorem{Defi}{Definition}[section]
\newtheorem{Rem}[Defi]{Remark}

\theoremstyle{plain}
\newtheorem{Prop}[Defi]{Proposition}
\newtheorem{Lemma}[Defi]{Lemma}
\newtheorem{Cor}[Defi]{Corollary}
\newtheorem{Thm}[Defi]{Theorem}

\newtheorem{Conj}[Defi]{Conjecture}
\newtheorem{Fact}[Defi]{Fact}

\title[Shimura curves in genus 3]{Shimura curves within the locus of hyperelliptic
Jacobians in  genus three}
\author{Samuel Grushevsky}
\address{Mathematics Department, Stony Brook University, Stony Brook, NY 11794-3651, USA.}
\email{sam@math.sunysb.edu}
\thanks{Research of the first author is supported in part by National Science Foundation under the grant DMS-12-01369.}
\author{Martin M\"oller}
\address{Institut f\"ur Mathematik, Goethe-Universit\"at Frankfurt, Robert-Mayer-Str. 6-8
60325 Frankfurt am Main, Germany}
\email{moeller@math.uni-frankfurt.de}
\thanks{Research of the second author is supported in part by ERC-StG 257137.}
\begin{document}
\begin{abstract}

We construct an infinite number of Shimura curves contained in the locus of hyperelliptic Jacobians of genus 3. In the opposite direction, we show that in genus 3 the only possible non-complete (in the moduli space of abelian threefolds) Kuga curves contained in the hyperelliptic locus have the same degeneration data as that of the examples we construct.

The locus of genus 3 hyperelliptic Jacobians is a divisor within the moduli space of principally polarized abelian threefolds, and our result demonstrates the techniques we develop more generally for dealing with Shimura curves contained within a divisor in the moduli space of abelian varieties.
\end{abstract}
\maketitle

\section{Introduction} \label{sec:intro}
Working over the field of complex numbers, we denote by $\AVmoduli$ the moduli stack of principally polarized
complex abelian varieties (ppav).
An  algebraic subvariety $C \to \AVmoduli=\SpgZ\backslash \SpgR/\mathop{U}(g)$
is called a {\em Shimura subvariety}  if the inclusion is induced by a homomorphism  $G \to \SpgR$ for
some algebraic group $G$ and if, moreover, $C$ contains a complex multiplication (CM) point.
Without the CM point assumption, the universal family of $C$ is called a {\em Kuga fiber space}
and so we refer to the subvariety $C \to  \AVmoduli$ as {\em Kuga variety}.
In particular a one-dimensional such subvariety is called a {\em Shimura curve} (resp.~{\em Kuga curve}).

We denote $\moduli$ the moduli stack of smooth curves of genus $g$, and recall that
taking the Jacobian of a smooth curve defines the Torelli morphism $\moduli\to\AVmoduli$, which is injective as a map of coarse moduli spaces.
There is the following well-known conjecture:
\begin{Conj} \label{conj:shimMg}
For $g$ sufficiently large there does not exist a Shimura subvariety of $\AVmoduli$ generically contained in the locus of Jacobians of smooth curves.
\end{Conj}
We refer to the comprehensive recent survey \cite{moor} of Moonen and Oort for the history, motivation, applications, and further discussion of this problem.

While at the moment we cannot address this general problem, in this paper we develop some techniques for studying this
problem for {\em non-complete} Kuga (or Shimura) curves.
The techniques we use to study Kuga curves are by
degeneration of modular forms on suitable toroidal compactifications $\barAVmoduli$ of $\AVmoduli$, and explicit constructions. The results we obtain are for the locus $\Hyp_3$ of Jacobians of smooth hyperelliptic curves of genus~$3$.

Our first main result is the following necessary condition.
\begin{Thm}\label{thm:main}
Any non-complete Kuga curve in $\AVmoduli[3]$ contained generically in the locus $\Hyp_3$ of Jacobians of smooth hyperelliptic curves of genus~$3$ must have degeneration of type $\sigma_{1+1}$ with equal orders of growth,
i.e.~after a suitable choice of a basis it must be given by period matrices of the form
$$
\left(\begin{matrix}
t&0&0\\
0&t&0\\
0&0&0
\end{matrix}\right)+M
$$
where $t\in\HH$ is the parameter, and the matrix $M$ is independent of $t$.
\end{Thm}
The types of degeneration are explained in Section~\ref{sec:coords}, and the notion of growth is
developed along with the general form of the period matrix of a Kuga curve in Section~\ref{sec:modonShimura}.
\begin{Rem}
We stress the result above only applies to
``non-complete'' Kuga curves --- those that are not complete in $\AVmoduli$. To be precise, recall that the locus of Jacobians of smooth curves is not closed in $\AVmoduli$, its closure contains the locus of Jacobians of stable nodal curves of compact type. Thus in particular the conjecture above is that there does not exist a Kuga curve contained in the closure in $\AVmoduli$ of the locus of Jacobians, but not contained in its boundary.

Kuga curves that lie entirely in the locus of Jacobians of smooth curves (without intersecting  the locus of decomposable abelian varieties $\Agdec$), are completely understood.
Without further hypothesis on completeness or non-completeness, one easily shows that
they are also Teichm\"uller curves, and hence non-complete. In fact, they are
also Kobayashi geodesics and since the Teichm\"uller metric and the
Kobayashi metric
coincide, one can use the argument in \cite[Theorem~4.1]{mcmullenbild} to
show this.
Such Shimura-Teichm\"uller curves are classified up
to possible exceptions in $g=5$ (\cite{moellerST}, see also \cite{MVZspecial}). There is only one
such curve for $g=3$, and one such curve for $g=4$. Both families are not
families of hyperelliptic curves.
\end{Rem}

In the opposite direction, we explicitly construct infinitely many Shimura curves:
\begin{Thm}\label{thm:construct}
For any fixed $u\in\QQ+i\QQ \setminus \ZZ+i\ZZ$, the family given
by period matrices of the form
\begin{equation}\label{gen3examples}
\Pi_{u}(t) = \left(\begin{matrix} t+iu^2&u^2/2&iu \\ u^2/2&t &u\\ iu&u& i \end{matrix}\right),
\end{equation}
where the parameter $t$ lies in the set
$\{\Im(z) > \max(\Re(u^2),0) \}$, defines a Shimura curve
whose generic point is contained in $\Hyp_3$.
\par
Moreover, there are infinitely many {\em distinct} Shimura curves among the ones given by such equations.
\end{Thm}
\begin{Rem}
The Shimura curve obtained above for the parameter
$u=\tfrac{1+i}{2}$ is in fact a family appearing in the list of Moonen and Oort
\cite{moor}, as we show in Proposition~\ref{prop:perX_t}. As discussed by Lu and Zuo \cite{luzuo}, this gives a
non-complete Shimura curve contained in the closure of the hyperelliptic locus in genus~$3$; these correspond to hyperelliptic curves with a reduced automorphism group $\ZZ_2\times\ZZ_4$.

We note that \eqref{gen3examples} gives liftings of Shimura curves to the Siegel space, which thus a priori may be equivalent under the action of the symplectic group. The second part of the theorem is thus the statement that there are infinitely many distinct Shimura curves in $\Hyp_3$, proven in
Proposition~\ref{prop:infdisjoint}.
\end{Rem}
\begin{Rem}
Independently, and working from a completely different viewpoint, Lu and Zuo \cite{luzuo} have just obtained results on non-existence of Kuga curves of so-called Mumford type, contained in $\cM_g$, for higher genus --- thus greatly advancing our knowledge towards conjecture \ref{conj:shimMg}. Our methods and results are in a sense complementary to those of Lu and Zuo \cite{luzuo}, and it would be interesting to see if some combination of their and our techniques may lead to further progress on the general problem. In a follow-up paper, we will use some geometric constructions of Pirola to construct an infinite number of non-complete Kuga curves contained in the locus of Jacobians of smooth curves of genus~$4$.
\end{Rem}

\subsection*{Outline of the paper}
Our proof of Theorem~\ref{thm:main} is by degeneration, by noticing that the closure in $\cA_3$ of the locus $\Hyp_3$ is given by one defining equation, the theta-null modular form, and studying the Taylor series (i.e.~the generalized Fourier-Jacobi expansion) of the theta-null near each boundary stratum of the toroidal compactification $\barAVmoduli[3]$ of $\cA_3$, restricted to a potential Kuga curve. Along the way of studying these restrictions, we naturally encounter the Kuga curves given by \eqref{gen3examples} (which turn out to actually be Shimura curves).

In Section~\ref{sec:back} we recall the notation for moduli of abelian varieties and modular forms. In Section~\ref{sec:coords} we recall and further describe the boundary strata of a toroidal compactification and analytic coordinates on the moduli spaces near them. In Section~\ref{sec:modonShimura} we set up the machinery of our approach, recalling and restating the result of \cite{vizu} on the structure of Kuga curves, and investigating the term-by-term vanishing of Fourier-Jacobi expansions. In Section~\ref{torusrank1} we study properties of theta functions on abelian surfaces and show that there do not exist Kuga curves contained in $\Hyp_3$ with degenerations of torus rank~1. In Section~\ref{torusrank3} we use a very different method --- the theorem of Mann on solutions of polynomial equations in roots of unity --- to show that there do not exist Kuga curves contained in $\Hyp_3$ with degenerations of torus rank~3. Finally, in Section~\ref{torusrank2} we study the possible Kuga curves with degenerations of torus rank 2, proving further restrictions, and constructing our infinite collection of examples. In Section~\ref{sec:family} we further identify one of these examples of Shimura curves as the locus of hyperelliptic curves with reduced automorphism group $\ZZ_2\times\ZZ_4$ from the list of Moonen and Oort \cite{moor}, as recently studied by Lu and Zuo \cite{luzuo}.

\subsection*{Acknowledgements}
We thank Klaus Hulek for useful discussions on Fourier series for modular forms. The first author would like to thank Igor Dolgachev for many stimulating discussions about special subvarieties of $\cM_g$. The authors are grateful to the Korea Institute for Advanced Studies (KIAS), where some of the work for this paper was done.

We are very grateful to Xin Lu and Kang Zuo for pointing out the example of a Shimura curve lying in $\Hyp_3$ to us, which then appeared in their recent paper \cite{luzuo}. This led us to uncover a mistake in the first version of this paper, and to further discover the infinitely many Shimura curves in $\Hyp_3$.

\section{Background and Notation} \label{sec:back}
Our approach to Kuga curves will be analytic, via the Fourier-Jacobi expansion of modular forms near various strata of the boundary of toroidal compactifications $\barAVmoduli$. We will now review the relevant notions from the analytic approach to the moduli of ppav and the characterization
of the hyperelliptic locus in genus three  using theta constants.

Denote
$$
  \HH_g:=\lbrace\tau\in\operatorname{Mat}_{g\times g}(\CC)\mid \tau^T=\tau, \operatorname{Im}\tau>0\rbrace
$$
(where $\tau^T$ denotes the transpose)
the Siegel upper half-space consisting of symmetric complex matrices
with positive-definite imaginary part. It is the universal cover
of the moduli stack $\cA_g$ of ppav, with the covering group $\SpgZ$ acting in the usual
way by
$$
\gamma\circ\tau:=(A\tau+B)(C\tau+D)^{-1}\quad {\rm for}\quad \gamma=\left(\begin{smallmatrix} A&B\\ C&D\end{smallmatrix}\right)\in\SpgZ.
$$
For future use we denote by $\cX_g\to\cA_g$ the universal family of ppav, which is globally the quotient $\cX_g=\SpgZ\rtimes\ZZ^{2g}\backslash\HH_g\times\CC^g$, where we think of $\ZZ^{2g}$ as acting by adding period vectors on a ppav. We also denote by $\cX_g^{\times r}:=\cX_g\times_{\cA_g}\ldots\times_{\cA_g} \cX_g$ the $r$-fold fiberwise product, i.e.~the generic fiber of the family $\cX_g^{\times r}\to\cA_g$ is the $r$'th Cartesian power of a ppav.

We call a ppav {\em decomposable} if it is equal to a product of lower-dimensional ppav (with polarization), and denote by $\Agdec\subset\AVmoduli$ the locus of such decomposable ppav.

A holomorphic function $F:\HH_g\to\CC$ is called a {\em Siegel modular form} of weight $k$ with respect to a subgroup $\Gamma\subset\SpgZ$ if
$$
  F(\gamma\circ\tau)=\det(C\tau+D)^k F(\tau)\qquad\forall \gamma\in\Gamma,\forall\tau\in\HH_g
$$
(additionally, for $g=1$ it needs to be assumed holomorphic at the cusps of the corresponding modular curve).

For $\eps,\delta\in (\ZZ/2\ZZ)^g$ the {\it theta function with characteristics} $\eps,\delta$ is defined by
\begin{equation}
\label{eq:thetaseries}
\thetalc{\eps}{\de}(\tau,z) :=
\tc{\eps}{\de}(\tau,z):=\sum\limits_{N\in\ZZ^g}e^{\pi i\left(N+\eps/2\right)^T\left(\tau(N+\eps/2)+z+\de/2\right)},
\end{equation}
where $\tau\in\HH_g$ and $z\in\CC^g$.
The corresponding {\em theta constant} is the evaluation of the theta function at $z=0$.
The characteristic $[\eps,\delta]$ is called {\em even/odd} depending on whether the scalar
product $\eps^T\cdot \de\in\ZZ/2\ZZ$ is $0$ or $1$, respectively. All theta constants with odd characteristics vanish identically in $\tau$, while all theta constants with even characteristics are in fact modular
forms of weight $1/2$ with respect to a certain finite index normal subgroup $\Gamma(4,8)\subset\SpgZ$. Under the action of the full group $\SpgZ$ the characteristics of theta constants are permuted under an affine action of $\SpgZ$ on $(\ZZ/2\ZZ)^{2g}$ that preserves parity, and moreover certain eighth roots
of unity arise, but by considering suitably symmetric polynomials in theta constants one can construct
modular forms for the entire group $\SpgZ$.

Perhaps the simplest such modular form is the so-called {\it theta-null}
$$\theta_{\nul}(\tau) := \prod_{[\eps,\delta]\,\,  \even} \tc{\eps}{\de}(\tau).$$

The basis for the current work are the following classical fact --- see \cite{ACGH} for the original references and ideas.

\begin{Fact}
In genus 3, the zero locus in $\AVmoduli[3]$ of the theta-null is equal to the closure of the locus $\Hyp_3$ of Jacobians of hyperelliptic genus~$3$ curves.
\end{Fact}
To put this paper in perspective, we mention that in genus four, the theta-null divisor is not
the divisor that cuts out the Schottky locus. On the contrary, the closure of the locus of Jacobians of genus~$4$ curves is the zero locus of the so-called Schottky form. We will use this, combined with the techniques developed below, to construct infinitely many Shimura curves contained in the locus of Jacobians of genus~$4$, in the follow-up paper.

Our approach to studying the hyperelliptic locus is by looking at its boundary, by expanding the defining equation for it in suitable Fourier series. We now recall this setup.

\section{Coordinates near the boundary of the toroidal compactification} \label{sec:coords}
In what follows we will consider the asymptotic behavior of modular forms near various boundary
strata of a suitable toroidal compactification $\barAVmoduli$. The general procedure for performing such an expansion follows from the construction of toroidal compactifications, where partial quotients of $\HH_g$ are taken first, and then glued, and further quotients are taken. Here we only summarize the computations as we need them, referring to \cite[II.4,V.1]{fach} for all details of the theory, to \cite{hukawebook} for many examples, and references and explanations of the structure in general; further examples are worked out in \cite[Prop.~IV.4]{husa} and \cite[Rem.~6.5]{grhusurvey}.

We recall that there exist many different toroidal compactifications each admitting a contracting morphism to the Satake compactification. Such toroidal compactifications are defined by choosing an appropriate $\GL_g(\ZZ)$-invariant fan in the cone of semipositive definite symmetric matrices with rational radical. However, for $g=3$ there exists only one ``minimal'' toroidal compactification, so that all other toroidal compactifications are obtained from it by further blowups. We denote this compactification $\barAVmoduli[3]$, which in this case is at the same time the perfect cone, second Voronoi, and central cone toroidal compactification. The corresponding fan contains one $\GL_3(\ZZ)$-orbit of maximal ($6$-dimensional) cones, and all other cones are obtained as orbits of its various faces. Each orbit of cones in the fan gives rise to a stratum of the boundary of $\barAVmoduli[3]$.

In Table~\ref{cap:listcones} we list all the cones in genus 3,
up to the action of $\GL_3(\ZZ)$, following the list of \cite[Section~5]{nakamura}, and the notation of \cite{vallentinthesis},\cite{huto}.
Note that in these sources, the cone $\sigma_{C_4}$
is given by the generators $y_1^2, y_2^2, (y_1 - y_3)^2,(y_2 - y_3)^2$, which is equivalent to the cone given below, by performing the change of variables $x_1 = y_2-y_3, x_2 = y_3-y_1, x_3=y_1$.
\begin{table}
\begin{tabular}{|c|c|c|}
\hline
&&\\
& Cone generators &Coordinates\\
&&\\
\hline
&&\\
$\sigma_1$& $x_1^2$ &  $T_1 = q_{11}$ \\
&&\\
\hline
&&\\
$\sigma_{1+1}$& $x_1^2, x_2^2$ &  $T_1 = q_{11}, T_2=q_{22}, S_1 = q_{12}$\\
&&\\
\hline
&&\\
$\sigma_{K_3}$& $x_1^2, x_2^2, (x_1-x_2)^2$ &  $T_1 = q_{11}q_{12},\,
T_2 = q_{22} q_{12},\, T_3 = q_{12}^{-1}$ \\
&&\\
\hline
&&\\
$\sigma_{1+1+1}$& $x_1^2, x_2^2, x_3^2$ &  $T_i = q_{ii},$ \quad
$S_1 = q_{12}, \, S_2 = q_{13}, \, S_3 = q_{23}$ \\
&&\\
\hline
&&\\
$\sigma_{K_3+1}$ & $x_1^2, x_2^2, (x_1 - x_2)^2, x_3^2$ &
$T_1 = q_{11}q_{12},\, T_2 = q_{22} q_{12},\, T_3 = q_{33},$ \\
&& $T_4 = q_{12}^{-1},\, S_1 = q_{13}, \, S_2 = q_{23}$ \\
&&\\
\hline
&&\\
$\sigma_{C_4}$ & $x_1^2, x_2^2, x_3^2,$  &
$T_1 = q_{11}q_{23}^{-1},\, T_2 = q_{22} q_{13}^{-1},\, T_3 = q_{33}q_{12}^{-1},$ \\
& $(x_1 + x_2 + x_3)^2$ & $T_4 = q_{23},\, S_1 = q_{12}q_{23}^{-1},\, S_2 = q_{13}q_{23}^{-1}$  \\
&&\\
\hline
&&\\
$\sigma_{K_4-1}$ & $x_1^2, x_2^2, x_3^2,$  & $T_1 = q_{11}q_{12}q_{13},\,T_2 = q_{22}q_{12},\, T_3 = q_{33} q_{13},$ \\
& $(x_1 - x_2)^2, (x_1 - x_3)^2$ & $T_4 = q_{12}^{-1}, \,T_5 = q_{13}^{-1}, \quad \quad S_1 = q_{23}$ \\
&&\\
\hline
&&\\
$\sigma_{K_4}$ & $x_1^2, x_2^2, x_3^2, (x_1 - x_2)^2$  &  $T_1 = q_{11}q_{12}q_{13}, \,
T_2 = q_{22}q_{12}q_{23},\, T_3 = q_{33}q_{13}q_{23}$ \\
& $(x_1 - x_3)^2, (x_2 - x_3)^2$ & $T_4 = q_{12}^{-1}, \, T_5 = q_{13}^{-1}, \, T_6 = q_{23}^{-1}$\\
&&\\
\hline
\end{tabular}
\vspace{5mm}
\caption{Cones in $\barAVmoduli[3]$, and the expressions for the unbounded coordinates $T_i$
and bounded coordinates $S_i$.}
\label{cap:listcones}
\end{table}

With each cone we associate its {\em dimension} $k$, i.e.~the number
of generators of the cone, and its {\em rank} $r$, i.e.~the rank
of a matrix in the interior of the cone. The rank is equal to
one for $\sigma_1$, two for $\sigma_{1+1}$ and $\sigma_{K_3}$, and
three for all the remaining cases in the table. The rank of a cone $\sigma$
equals the rank of the image in the Satake compactification
of the boundary stratum corresponding to  $\sigma$.

The following result is a restatement of well-known facts about
toroi\-dal compactifications (see e.g.~\cite{fach}) and a
consequence of the smoothness (\cite{igusadesing})
of the stack $\barAVmoduli[3]$. Let $p\in\partial\barAVmoduli[3]$ be a
boundary point, and let $\delta$ be the  boundary stratum
in which $p$ is generic and let $r = r(\delta)$ be its rank.

\begin{Prop}\label{prop:coords}
Any point $p\in\delta\subset\partial\barAVmoduli[3]$ as above has a neighborhood that
is a quotient of a neighborhood of the point
$$ (s_{1,0}, \ldots, s_{\ell,0}, 0, \ldots,0, Z, z_{11},\ldots,z_{r(g-r)}) \in
(\CC^*)^\ell \times \CC^k \times \cX_{g-r}^{\times r}$$
under an unramified covering map, where $S_j$, $T_j$ are coordinates
in $(\CC^*)^\ell$ and $\CC^k$ respectively, $Z\in\HH_{g-r}$ and
$\ell={r(r+1)/2-k}$.

The change of variables between the coordinates $q_{ij}:=e^{2\pi i\tau_{ij}} $ for
$i,j \leq r$ and the $S_i$ and $T_i$ is shown in the last column
of Table~\ref{cap:listcones}.
\end{Prop}
We refer to the $T_j$ as {\it unbounded} variables (since they are
allowed to go to zero), to the $S_i$ as the {\it bounded} variables,
and to $(Z,z_{ij})$ as {\it moduli coordinates}.

\section{Modular forms vanishing on Kuga curves} \label{sec:modonShimura}

A {\em Kuga fiber space} is an inclusion $j$ of a $\QQ$-algebraic group
$G$ into $\Sp_{2g}(\QQ)$, such that an arithmetic lattice $\Gamma \subset G_\RR$
maps to $\Sp_{2g}(\ZZ)$ and such that the intersection of the maximal compact
subgroup $U_g \subset \Sp_{2g}(\RR)$ with $j(G_\RR)$ is a maximal compact
subgroup $K_\RR$ of $G_\RR$. We identify a Kuga fiber space with the
subvariety of the moduli stack
$$ \Gamma \backslash G_\RR/K_\RR \hookrightarrow \Sp_{2g}(\ZZ) \backslash \Sp_{2g}(\RR)/U_g = \AVmoduli$$
(The term Kuga fiber space originally refers to the induced family of abelian varieties over
$ \Gamma \backslash G_\RR/K_\RR$, which of course determines the above data and vice versa).
A {\em Kuga curve} is a Kuga fiber space of (base)
dimension one. A {\em Shimura subvariety} of $\AVmoduli$ is the image of such an inclusion
that moreover contains a CM point. While Shimura curves are more commonly studied, the natural context for our results, similarly to \cite{luzuo}, is Kuga curves.

\smallskip
Our technique is by studying the degenerations of Kuga curves, and begins with the following structural result giving a geometric description of non-complete Kuga curves due to Viehweg and Zuo.
\begin{Thm}[{\cite[Theorem~0.2]{vizu}}] \label{thm:vztheorem}
Given any non-complete Kuga curve $C\subset\AVmoduli$
parameterizing a family of
ppav $\cX \to C$, there exists an \'etale cover $C_0\to C$  such that
the pullback family $\cX_0\to C_0$ is isogenous, over $C_0$, to
a product
\begin{equation}\label{eq:productfamily}
F \times (\cE_0 \times_{C_0} \times \cdots \times_{C_0} \cE_0),
\end{equation}
where $[F]\in\AVmoduli[g-r]$ (as a constant family over $C_0$),
and the product is taken with the $r$'th power of the pullback of
the universal family of elliptic curves $\cE\to\AVmoduli[1]$
under some unramified covering $C_0 \to \AVmoduli[1]=\SL_2(\ZZ)\backslash\HH$.
\end{Thm}
\begin{proof}
This theorem is stated in \cite{vizu} for families
that attain the Arakelov bound. That families over Kuga curves
attain the Arakelov bound is implicit in many sources, and is explicitly stated e.g.~in \cite{MVZ}.
\end{proof}
Let $\overline{C}$ be the closure in $\barAVmoduli$ of the curve $C$ above.
Then, for a boundary point $p\in \overline{C}\cap\partial\barAVmoduli$
the toric part of the corresponding semiabelic variety has dimension $r$ (this rank is the same for all boundary points $p$).
We thus call $r$ the {\em torus rank} associated to a
curve $C$ and call $F$ the {\em fixed part} of (the family of abelian
varieties over) the Kuga curve. If $C$ is of torus rank $r$, then
all boundary points of $C$ lie in strata of rank $r$.
In terms of variations of Hodge structures, the maximal unitary subsystem
of $\VV = R^1 f_* \CC_\cX$ has rank $2(g-r)$.
\par
Using this criterion, we can conjugate
the period matrix of a Kuga curve to a standard form.

\begin{Lemma} \label{le:VZmatrixform}
Any non-complete Kuga curve $C\to\AVmoduli$ with a degeneration of torus rank $r$ has a
lifting $\varphi:\HH\to\HH_g$ (where $\HH\to C$ is the universal cover) given for $t\in\HH$ by
\begin{equation}\label{eq:varphi}
\varphi(t) =
A Z_t A^T \, + BA^T,
\quad \text{where} \quad
Z_t: =
\left( \begin{matrix}
t & \cdots & 0  & 0 \\
\vdots & \ddots & \vdots & \vdots \\
0 & \cdots & t & 0 \\
0 &  \cdots & 0 & Z\\ \end{matrix} \right)
\end{equation}
for some $A \in \GL_g(\QQ)$ and some $B \in \operatorname{Mat}_{g\times g}(\QQ)$ such that $BA^T$ is symmetric,
where there are $r$ copies of $t$ on the diagonal,
and $Z\in\HH_{g-r}$ is independent of $t$.
\end{Lemma}
\begin{proof}
The maps $C \to \AVmoduli[g]$ and the corresponding map $C_0 \to \AVmoduli[g]$
from Theorem~\ref{thm:vztheorem} lift to the same map of
universal covers $\HH \to \HH_g$, and we can thus use the product
decomposition~\eqref{eq:productfamily} to describe the behavior of period matrices
along any Kuga curve.

The block diagonal matrix $Z_t$ is the period matrix of the product family
in \eqref{eq:productfamily} with respect to some symplectic basis
$\{\tilde{\alpha}_i, \tilde{\beta}_i,i=1,\ldots,g\}$ of $H^1(B\times E_0^r,\ZZ)$,
where $E_0^r$ is some fiber of $\cE_0^r$. We let $X_0$ be the corresponding
fiber of $\cX_0$.
The rational representation $\rho_r$ of the isogeny $F \times \cE_0^r \to \cX_0$
defines a symplectic isomorphism $$\rho_r: H^1(F\times E_0^r,\QQ) \to H^1(X_0,\QQ).$$
The image $L = \rho_r(\langle \tilde{\alpha}_i,i=1,\ldots,g \rangle)$ is
a Lagrangian subspace of $H^1(A_0,\QQ)$. If we pick a basis $\alpha_1,\ldots,\alpha_g$
of $L \cap H^1(X_0,\ZZ)$ and complete it by $\beta_1,\ldots,\beta_g$ to a
basis of $H^1(X_0,\ZZ)$, then $\rho_r$ in the basis of $H^1(F\times E_0^r,\QQ)$
and $H^1(X_0,\ZZ)$ just specified is given by some block
matrix $\left( \begin{smallmatrix} A & B \\ 0 & D \\
\end{smallmatrix} \right) \in \SpgQ$.  This matrix being symplectic implies $D=A^{-T}$, and thus
the period matrix of $X_0$ is obtained from that of $B\times E_0^r$ as
claimed in \eqref{eq:varphi}. Over the contractible space $\HH_g$
the choice of a basis can of course be made consistently in families.
\end{proof}
We can write the equation for this period matrix still more explicitly.
\begin{Lemma} \label{le:V2ndZmatrixform}
Any non-complete Kuga curve $C\to\AVmoduli$ with a degeneration of torus rank $r$, such that $\overline C\subset\barAVmoduli$ intersects the stratum corresponding to a cone $\sigma$ in the fan defining $\barAVmoduli$ has a lifting $\varphi:\HH\to\HH_g$ given by
\begin{equation} \label{eq:varphinew}
\varphi(t) =  \left(\begin{matrix} E\cdot t & 0 \\ 0& 0 \end{matrix}\right) +
A\left(\begin{matrix} 0 & 0 \\ 0& Z\end{matrix}\right) A^T  +R,
\end{equation}
where
$E=(e_{ij})$  lies in the chosen representative of
the $\GL_r(\ZZ)$ orbit of the cone $\sigma$,
while $Z\in \HH_{g-r}$, $A \in \GL_g(\QQ)$, and symmetric $R \in \operatorname{Mat}_{g\times g}(\QQ)$ are fixed independently of $t$.
\end{Lemma}
\begin{proof}
The formula for $\varphi(t)$ is just the restatement of the previous lemma.
Since $E$ lies in a cone of the fan that lies in the orbit of
$\sigma$ under $\GL_r(\ZZ)$ (this is what it means that the degeneration corresponds to the stratum $\sigma$),
there exists some $A_0 \in \GL_r(\ZZ)$ such that
$A_0 E A_0^T$ lies in the chosen representative of the orbit of $\sigma$ under $\GL_r(\ZZ)$.
Writing $A=\left( \begin{smallmatrix} A_0 & 0 \\ 0 & Id \\
\end{smallmatrix} \right) \in \GL_g(\ZZ)$,
this corresponds to the action of $\left( \begin{smallmatrix} A & 0 \\ 0 & A^{-T} \\
\end{smallmatrix} \right) \in \SpgZ$, thus to the choice of an appropriate lifting
of the Kuga curve.
\end{proof}
\par
We refer to the coefficients $e_{ij}$ of $E$ as the {\em growth orders} of the Kuga curve.
Their importance will soon become apparent when we study modular forms vanishing on
a Kuga curve.

We first give an even more precise form of this period matrix in the case that will be
the most interesting in the sequel.

\begin{Cor} \label{cor:AZAterm}
In the notation of the preceding lemmas, if $g=3$ and $r=2$, then we may moreover
suppose that the term $A\left(\begin{smallmatrix} 0 & 0 \\ 0& Z_z
\end{smallmatrix}\right) A^T $ in \eqref{eq:varphinew} can be written as
$K Z_{11}$ with $Z_{11}$ in a fundamental domain $\cF$ for the action of $\SL_2(\ZZ)$ on $\HH$, and $K$
a rational symmetric matrix of rank one with $K_{33} = 1$, i.e. with
$$K = \left(\begin{smallmatrix} a^2 & ab & a \\
ab & b^2 & b \\
a & b & 1 \\
\end{smallmatrix}\right), \quad a,b \in \QQ. $$
\end{Cor}
\begin{proof} The term in question is, to start with, symmetric of rank one and
a rational multiple of the matrix $Z$, which is just a number $Z_{11}$ by the dimension
hypothesis. Absorbing $K_{33}$ into the constant, we thus get a matrix of rank 1, which is thus of the form claimed.

It remains to show that by the continued fraction algorithm we may suppose
$Z_{11} \in \cF$ without destroying the rest of the shape of $\varphi(t)$.
Translating just changes $R$, so all we need is to control the effect of applying
$$ S = \left(\begin{smallmatrix}
1 & 0 & 0 & 0 & 0 & 0\\
0 & 1 & 0 & 0 & 0 & 0\\
0 & 0 & 0 & 0 & 0 & 1\\
0 & 0 & 0 & 1 & 0 & 0\\
0 & 0 & 0 & 0 & 1 & 0\\
0 & 0 & -1 & 0 & 0 & 0\\
\end{smallmatrix}\right) \in \Sp_6(\ZZ). $$
A short calculation shows that $\varphi(t)$ still has the three summands
$\left(\begin{smallmatrix} E\cdot t & 0 \\ 0& 0 \end{smallmatrix}\right)$, a (different)
rational symmetric matrix $R$, and a rational symmetric rank one multiple of
$-1/Z_{11}$.
\end{proof}
\par
\begin{Rem} \label{rem:shimura_crit}
Conversely, given a family of abelian varieties with period matrices
as defined in \eqref{eq:varphinew} for $t \in \HH$,  this family is a Kuga curve.
In fact, the family $\varphi_0(t)= Z_t$ with $Z_t$ as in \eqref{eq:varphi}
defines a Kuga curve, being a Kuga curve is preserved by an isogeny
and $\varphi_0(t)$ and $\varphi(t)$ differ by the effect of an isogeny, as
explained above.
\end{Rem}
We will now prove a necessary criterion for the existence of a non-complete Kuga curve
in $\AVmoduli$, contained in the zero locus of a modular form $F$, such that
its closure in $\barAVmoduli$ contains a point in a stratum corresponding to a cone
$\sigma$ of the fan defining $\barAVmoduli$. Basically the condition is that the Taylor series of $F$ vanishes along
$\overline{C}$ term by term near the boundary.

We will denote $F^\sigma$ this Taylor expansion of a modular form $F$ near a boundary stratum $\sigma$ in coordinates provided by Proposition \ref{prop:coords} --- this is called a (generalized) Fourier-Jacobi expansion. We think of it as a power series in the unbounded variables $T_i$, with coefficients being functions of the bounded variables $S_i$ and the moduli coordinates.

We then substitute into such a Fourier-Jacobi expansion $F^\sigma$ the period matrix given
by~\eqref{eq:varphinew}, and then order the terms according to the resulting power
of $q:=e^{2\pi i t}$. While this ordering in general depends on the growth orders $e_{ij}$ in
Lemma \ref{le:V2ndZmatrixform}, it turns out that in all our cases the modular form, as a power series in all of the $T_i$, has a unique lowest order
term (l.o.t.). Indeed, we introduce a partial ordering on monomials in $T_i$ by
saying that a monomial is {\it less than or equal} to another one if it
divides it --- that is, if the power of each
variable $T_i$ is less than or equal. We then have the following criterion (stated for maximal torus rank for clarity, see the comment below for the general case)
\begin{Prop} \label{pr:criclowestterm}
Suppose for some  cone $\sigma$ of torus rank $g$, in a fan
defining a smooth toroidal compactification $\barAVmoduli$,
the Fourier-Jacobi expansion
$F^\sigma$ of a modular form $F$ has a unique lowest
order term, which is a monomial in $T_i$, multiplied by some function
$F_0^\sigma(S_1,\ldots,S_\ell)$ of the bounded variables.

Then, if there exists a Kuga curve $C \to \AVmoduli$ contained in
the zero locus of $F$, such that its closure  $\overline{C}\subset\barAVmoduli$ intersects the stratum corresponding to $\sigma$, then there must exist roots of unity
$\zeta_i$ such that
$$F_0^\sigma(\zeta_1,\ldots,\zeta_\ell)= 0.$$
\end{Prop}
\begin{Rem}
We will use the full strength of this proposition, including the statement about roots of unity, to deal with the case of cones of torus rank 3 for $\AVmoduli[3]$ in section \ref{torusrank3}. Notice that there is of course a more general version of the statement above, for $\sigma$ of arbitrary torus rank $r$. In that case the lowest order term $F_0^\sigma$ would also depend on the moduli coordinates $Z,z\in\cX_{g-r}^{\times r}$, and the statement is that that we would have $F_0^\sigma(S_1, \ldots,S_\ell,z,Z)= 0$ where $S_i$ and $z$ all depend on $Z\in\HH_{g-r}$, $A \in \GL_g(\QQ)$, and $R \in \operatorname{Mat}_{g\times g}(\QQ)$ as appearing in Lemma \ref{le:V2ndZmatrixform}. Instead of writing out this statement in detail in full generality, we will prefer to work out explicitly what this says for our situations for cones of torus rank 1 and 2.
\end{Rem}
\begin{proof}[Proof of Proposition \ref{pr:criclowestterm}]
As a consequence of Lemma~\ref{le:V2ndZmatrixform}, denoting from now on $q:=e^{2\pi i t}$, we have
$q_{ij} = e^{2\pi i b_{ij}} q^{e_{ij}}$.
Consequently, each $T_i$ is a rational power of $q$ times a root
of unity. To study the behavior of each $T_i$ near $q=0$, i.e.~as $t\to i\infty$,
we divide $t$ by the least common multiple of all denominators of $b_{ij}$ (which
only reparameterizes the curve), and can thus assume that we have
$T_i = \xi_i q^{f_i}$ for each $i$, with $f_i \in \NN$ and $\xi_i$ a root of unity. The same
applies to the $S_i$, but since they stay bounded away
from zero by their definition and the choice of the cone, the exponent of $q$
is zero, i.e.~$S_j = \zeta_j$ are all roots of unity.

If $F$ vanishes identically on $C$, then its Fourier-Jacobi expansion $F^\sigma$ must vanish identically as a power series in $q$, and thus in particular the lowest order term in $q$, which is the evaluation of $F^\sigma_0$ at the $S_i$ as above, must vanish.
\end{proof}
Our proof of the main result will proceed differently for different torus ranks $r$. In general, the criterion above seems more powerful for larger $r$, as then there are fewer variables $S_i$, and the condition for $F_0^\sigma$ to vanish at some roots of unity becomes most restrictive. Notice also that for $r<g$ the variables $Z,z$ enter the expression, and thus the vanishing of l.o.t.~and of consequent minimal terms would in fact imply certain vanishing statements for theta functions of $Z,z$, and we will need to study the associated geometry in more detail --- while for $r=g$ the question becomes purely about roots of unity, and is essentially combinatorial.

\section{Non-existence of Kuga curves of torus rank 1} \label{torusrank1}
In this section we prove that there do not exist non-complete Kuga curves whose generic point is contained in $\Hyp_3$, and such that their degeneration has torus rank one.
This is to say, we prove that there do not exist non-complete Kuga curves $C\subset\Hyp_3$ such that their closure in $\barAVmoduli[3]$ intersects the stratum $\sigma_1$. The argument, here as well as in the following sections, is by substituting the expression for the Kuga curve given by Lemma \ref{le:V2ndZmatrixform} and applying Proposition \ref{pr:criclowestterm}.

The results in this section, on torus rank 1, are in fact not new, as they are a
consequence of a more general result of Xiao \cite{XiaoHyp}. Indeed, after possibly
passing to a ramified cover of the base, from a Kuga curve degenerating to $\sigma_1$ we obtain a family of curves
$f: \cY \to C$ (of which $\cX$ is the family of Jacobians). If $C$ is not
contained in $\Agdec[3]$, then a generic fiber of $f$ is smooth.

It then follows from the work of Xiao \cite{XiaoHyp} that the dimension $q_f$ of the
fixed part of any family $f: \cY \to C$ of hyperelliptic curves of genus $g$ is bounded by
$$  q_f \leq (g-1)/d + 1,$$
where $d$ is the degree of the Albanese map $\cY \to \operatorname{Alb}(\cY)$ of the
fibred surface $\cX$.
Hence for genus three a fixed part of dimension $q_f=2$ is possible only
if $d=2$, in which case the upper bound is attained. Xiao moreover proves that
if the upper bound is attained, then the family is isotrivial, which is impossible for a Kuga curve.

\smallskip
Still, we obtain an independent proof, which also allows us to demonstrate our technique for the simplest case of the Fourier expansion. Recall that we only have one cone $\sigma_1$ of torus rank 1 in table \ref{cap:listcones}, generated simply by
the quadratic form $x_1^2$, with the corresponding unbounded variable being $T_1=q_{11}$, and the Fourier-Jacobi expansion is the usual one.
Indeed, by definition, for any $Z \in \HH_{g-1}$, any $z,b \in \CC^{g-1}$, and any $x \in \CC$ we have
\begin{flalign}
& \thetat{0& \e}{\de_1& \de} \left( \left(\begin{matrix} t&z^T\\ z& Z \end{matrix}\right),
 \left(\begin{matrix} x \\ b\end{matrix}\right) \right) =\sum_{N\in\ZZ}q_{11}^{N^2/2}w^N(-1)^{N\de_1}\thetalc\e\de(Z,b+Nz) & \nonumber \\
 & =\thetalc\e\de(Z,b)\,+\,q_{11}^{1/2}\,(-1)^{\de_1} \left(w\, \thetalc\e\de(Z,b+z) + \, w^{-1} \thetalc\e\de(Z,b-z) \right)   \nonumber \\
& \quad\quad\quad\quad\quad\,\,\,\, +O(q_{11}^{2})\,,\label{eq:fe_eps0} &
\end{flalign}
if the first top characteristic is zero, while
\begin{flalign}
&\thetat{1& \e}{\de_1& \de} \left( \left(\begin{matrix} t&z^T\\ z& Z \end{matrix}\right),
 \left(\begin{matrix} x \\ b\end{matrix}\right) \right)& \nonumber\\
& =\sum_{N\in\ZZ}q_{11}^{(2N+1)^2/8}w^{(2N+1)/2}i^{(2N+1)\de_1}\thetalc\e\de(Z,b+(2N+1)z/2)& \nonumber\\
&=q_{11}^{1/8}\left(i^{\de_1} \,\,w^{\frac12}\,\thetalc\e\de\left(Z,b\! +\!\frac{z}2\right)\,+\,
i^{-\de_1} \,\,w^{-\frac12}\,\thetalc\e\de \left(Z,b\! -\! \frac{z}2\right)
\right) & \nonumber\\
&+  q_{11}^{9/8}\left(i^{-\delta_1}w^{\frac32} \thetalc\e\de \left(\tau,b + \frac{3z}2\right)\, +\,
i^{\delta_1}w^{-\frac32} \thetalc\e\de\left(\tau,b - \frac{3z}2\right)
\right) & \nonumber \\
&  +\,O(q_{11}^{25/8}).\label{eq:fe_eps1}
\end{flalign}
where $w: = \exp(2\pi i x)$.

We can now prove the result for torus rank 1.
\begin{Prop}
There does not exist a non-complete Shimura curve generically contained in $\Hyp_3$ whose closure $\overline{C}\subset\barAVmoduli[3]$ intersects the stratum corresponding to $\sigma_1$, the locus of semiabelic varieties of torus rank one.
\end{Prop}
\begin{proof}
Since $\theta_\nul$ is a product of theta constants, for it to vanish identically on $\overline{C}$, one theta constant needs to vanish identically. The Fourier-Jacobi expansion corresponding to the cone
$\sigma_1$ is the classical one, given for theta constants by formulas \eqref{eq:fe_eps0}, \eqref{eq:fe_eps1} evaluated at $(z,b)=(0,0)$. Thus we need to check that it is impossible for the Fourier-Jacobi series of one theta constant $\thetalc\e\de$ to vanish identically
identically along the Kuga curve given by $\varphi(t)$ from formula \eqref{eq:varphinew} in Lemma~\ref{le:V2ndZmatrixform}.

The l.o.t.~of the Fourier-Jacobi expansion of a theta constant is well-known (see for example \cite{grsmconjectures} and references therein), and we recall the situation now, distinguishing the cases of $\e_1=0$ and $\e_1=1$. The l.o.t.~for a theta constant with  $\e_1=0$ is simply $\thetalc\e\de(Z,0)$ (the theta constant in dimension 2). For this to vanish for $Z\in\HH_2$ implies that $Z$ lies on the theta null divisor in genus 2, which is to say $Z\in\Agdec[2]$. By acting by $\mathop{Sp}_4(\ZZ)$, we can then conjugate $Z$ to the block-diagonal form (so $\tau_{23}=0$ and $q_{23}=1$), and the vanishing theta constant must then be for $\e=\de=11$. Substituting this in the Fourier-Jacobi expansion yields a vanishing constant term, while the next term has the form $$2(-1)^{\delta_1}q_{11}^{1/2}\thetat{11}{11}(Z,z)=
 2(-1)^{\delta_1}q_{11}^{1/2}\thetat{1}{1}(Z_1,z_1)\thetat{1}{1}(Z_2,z_2),
$$
since the theta function on a product of ppav is the product of theta functions on the factors. Thus for this next term to vanish, one of the two factors on the right hand side has to vanish. Since on an elliptic curve the only zero of $\thetalc{1}{1}$ is at the origin, we must have either $z_1=0$ or $z_2=0$, which is to say either $\tau_{12}=0$ or $\tau_{13}=0$, and in either case the period matrix $\varphi(t)$ of any point on the Kuga curve is decomposable.

\smallskip
Similarly, for a theta constant with $\e_1=1$ the terms in the Fourier-Jacobi expansion are non-zero multiples of $\thetalc\e\de(Z,(N+1/2)z)$ for $N\in\NN$, which  must then all vanish. Consider the functions $f_1(x):=\thetalc\e\de(Z,x+z/2)$ and $f_2(x):=\thetalc\e\de(Z,x+3z/2)$. For either of them, the locus of $x$ where they vanish is a translate of the theta divisor on the abelian surface. However, both $f_1$ and $f_2$ vanish for $x=0,z,2z$. If $Z\not\in\Agdec[2]$, then the theta divisor on it is the genus two curve, of which the abelian surface is the Jacobian, and two different translates of the theta divisor intersect in precisely two points; thus we must have $0=2z$, so $z$ is a 2-torsion point. In this case the only condition we have is $\thetalc\e\de(Z,z/2)=0$ (since $\thetalc\e\de(Z,3z/2)=\thetalc\e\de(Z,-z/2)=0$ is automatic by parity of the theta function). It turns out that for $Z\in\cA_2$ this in fact implies that $z=0$, this is proven in \cite[Proposition~2.1c]{grant}.

For completeness, we give a direct proof as follows: it is known that on an indecomposable abelian surface the theta divisor passes through precisely 6 (odd) 2-torsion points, and the intersection of the theta divisor with any its translate by a 2-torsion point contains 2 of these 6 points (this can be phrased as an azygy condition, and easily counted by hand, being the odd 2-torsion points such that their sum with the translate is still odd). Since the self-intersection number of the theta divisor is equal to 2, the intersection of the theta divisor with its translate by a 2-torsion point must consist only of these two points, for any indecomposable abelian surface. However, in our situation the theta divisor translated by $(\e Z+\de)/2$, and the theta divisor translated by $(\e Z+\de)/2+z$ intersect at $z/2=-3z/2$, and thus $z/2$ must be a 2-torsion point, so $z=0$, as a point on the abelian surface.

Thus we must have $Z\in\Agdec[2]$, so the theta function factorizes, and in particular the next order term is a non-zero multiple of
$$
 \thetat{\e_2 \e_3}{\de_2 \de_3}(Z,z/2) =\thetat{\e_2}{\de_2}(Z_1,z_1/2)\thetat{\e_3}{\de_3}(Z_2,z_2/2),
$$
for which to vanish either $z_1/2$ or $z_2/2$ must be a 2-torsion point, so that either $z_1=0$ or $z_2=0$, and then $\varphi(t)\in\Agdec[3]$ for all $t$.
\end{proof}

\section{Non-existence of Kuga curves of torus rank 3} \label{torusrank3}
In this section we show that there do not exist non-complete Kuga curves $C\subset\Hyp_3$ such that their closure in $\barAVmoduli[3]$ contains a semiabelic variety of torus rank 3. Here, very unlike the previous situation, the geometry of the theta function becomes irrelevant, as the corresponding Fourier expansion is now in all variables, with no abelian part left. However, the situation is complicated by the fact that there are now multiple cones of torus rank~$3$, as listed in Table~\ref{cap:listcones}. The hardest case here is that of $\sigma_{1+1+1}$, where in the Fourier-Jacobi the l.o.t.~is a complicated rational function, and to deal with it we will need to use the full strength of Lemma~\ref{le:V2ndZmatrixform}, in particular the rationality of the coefficients in \eqref{eq:varphinew}, which leads to $q_{jk}=e^{2\pi i b_{jk}}$ for $j\ne k$ being roots of unity. Indeed, we apply a theorem of Mann on solutions of polynomial equations in roots of unity, the setup for which we now recall.

An equation $\sum_{i=1}^{k} a_i \,\zeta_i = 0$ where the $\zeta_i$ are pairwise different roots of unity and where the $a_i$ lie in some number field $K$ is called a {\em $K$-relation
of length $k$}. Such a relation is called {\em irreducible} if
there does not exist a non-empty subset $I \subsetneq \{1,\ldots,k\}$
with the property $\sum_{i \in I} a_i\, \zeta_i = 0$.
Thus each $K$-relation is a sum of irreducible relations, though there may be
more than one way of writing a relation as a sum of irreducibles. We
will use the following bound on possible roots of unity appearing in a solution.

\begin{Thm}[Mann \cite{mann}] \label{thm:mann}
Suppose that $\sum_{i=1}^{k} a_i \zeta_i = 0$ is an irreducible $\QQ$-relation. Let $2=p_1 <3= p_2 < \cdots <p_s \leq k$ be the primes
listed in increasing order. Then if there exist roots of unity $\zeta_i$ satisfying the relation, then for some $N$ dividing the product $\prod_{i=1}^s p_i$ there exists a primitive $N$-th root of
unity $\zeta_N$, and there exist exponents $b_i \in \ZZ$ such
that
$$\zeta_i = \xi \,\zeta_N^{b_i},$$
for all $i=1,\ldots, k$, where $\xi$ is an arbitrary root of unity (of any degree).
\end{Thm}
\smallskip
We will now investigate the Fourier-Jacobi expansion corresponding to
the cone $\sigma_{1+1+1}$. Recall that this cone is generated by $x_1^2,x_2^2,x_3^2$,
the corresponding stratum is of codimension 3 in $\barAVmoduli[3]$, and the
unbounded and bounded variables are simply the diagonal and off-diagonal $q_{ij}$, respectively.
It will turn out that the cases of the other torus rank 3 strata will largely be
obtainable from this one by a further degeneration.

Computationally it turns out to be advantageous (for factorization below) to use
$\tilde q_{ij}:=q_{ij}^{1/2}$ instead of the bounded coordinates $S_k$, and we will thus switch to this notation.
Determining the l.o.t.~of each individual theta constant yields the l.o.t.~of $\theta_\nul$ in the Fourier-Jacobi expansion for $\sigma_{1+1+1}$.
\begin{Lemma} \label{le:coeff444}
The l.o.t.~of $\theta_\nul$ is equal to
$T_1^2 T_2^2 T_3^2$ times the coefficient
\begin{equation}\begin{aligned}\label{eq:coeff444}
L:=&(\tilde q_{23}^2 - 1)^2  (\tilde q_{13}^2 - 1)^2(\tilde q_{12} - 1)^2
\tilde q_{23}^{-4} \tilde q_{13}^{-4}  \tilde q_{12}^{-4}  \cdot \\
&  \cdot (\cq_{12}\cq_{13}\cq_{23} - \cq_{12} - \cq_{13} + \cq_{23})
(\cq_{12}\cq_{13}\cq_{23} - \cq_{12} + \cq_{13} - \cq_{23})   \cdot \\
&  \cdot  (\cq_{12}\cq_{13}\cq_{23} + \cq_{12} - \cq_{13} - \cq_{23})
(\cq_{12}\cq_{13}\cq_{23} + \cq_{12} + \cq_{13} + \cq_{23})
\end{aligned}\end{equation}
\end{Lemma}
Before proceeding to the proof, we prepare notation for discussing the l.o.t.~of
theta constants in general. Note that in general the monomials appearing in the
series for $\thetalc\e\de$ are of the form
$ q_{11}^{a_1}q_{22}^{a_2}q_{33}^{a_3}q_{23}^{b_1}q_{13}^{b_2}q_{12}^{b_3}$.
where we have
\begin{equation}\label{ab}
a_i = \frac12(n_i+\e_i/2)^2, \quad b_i = (n_j + \e_j/2)(n_k +\e_k/2)
\end{equation}
for $\lbrace i,j,k\rbrace =\lbrace 1,2,3\rbrace$.

\begin{proof}[Proof of Lemma~\ref{le:coeff444}]
To prove that $\theta_\nul$ has a unique l.o.t. in the Fourier-Jacobi expansion corresponding
to $\sigma_{1+1+1}$, it suffices to prove that
each individual theta constant has a unique l.o.t.~in the
partial ordering in $T_i$.
The quadratic form $a_i$ in the variables $n_i$ has a unique minimum
for each choice of $\e_i$ and this minimum is attained for finitely many
values of $n_i$, more precisely for $n_i = 0$ if $\e_i=0$ or for
$n_i \in \{0,-1\}$ if $\e_i = 1$. For
each $i$ there are $16$ theta characteristics where in the l.o.t.~$a_i = 1/8$,
while for the remaining $20$ theta characteristics  for the l.o.t.~$a_i=0$.

This proves the uniqueness of the l.o.t.; to compute it, consider for example
the product of all theta constants with $(\e_1,\e_2,\e_3)=(1,1,0)$ and $\de_i$
arbitrary.  The coefficient of the l.o.t.~in the product of the corresponding
theta constants, up to a root of unity, is equal to
$(\cq_{12}-\cq_{12}^{-1})^2(\cq_{12}+\cq_{12}^{-1})^2$, where the two summands
correspond to the choices for $n_i$ that are possible to attain
the minimum order in the variables $T_i$, as discussed above, and
the signs in the two factors correspond to the choice of $(\delta_1,\delta_2)$,
and the two remaining choices for $\delta_3$ result in the global square.

A similar discussion leads to the factors $(\tilde q_{ij} - 1)^2
\tilde q_{ij}^{-4}$ for the other characteristics with two $\e_i$ being equal to
one. The factors with $4$ summands correspond to $(\e_1,\e_2,\e_3)=(1,1,1)$;
there are $4$ of them, since there are $4$ possible choices for
$(\delta_1,\delta_2,\delta_3)$ in this case so that the characteristics is even.
\end{proof}
Thus for $\theta_\nul$ to vanish on a Kuga curve degenerating to a point
contained in the stratum corresponding to $\sigma_{1+1+1}$, the coefficient
$L$ above has to vanish. Suppose we have $\cq_{jk}=1$ (we will later show that the complicated
factors in \eqref{eq:coeff444} cannot vanish otherwise for $\cq$ being roots of unity) --- then we will
need to look at the next lowest order term. Without loss of generality, we assume that
$q_{12}=1$, and evaluate $\theta_\nul$ then.
\begin{Lemma} \label{le:coeffq12is0}
In the expansion of the power series $$ q_{11}^{-2} q_{22}^{-2}q_{33}^{-2} \, \theta_\nul|_{q_{12}=1} \in \CC(S_2,S_1)[[T_1,T_2,T_3]]$$
every non-zero term is divisible by  $q_{33}=T_3$.  More precisely, the unique l.o.t.~of $\theta_\nul|_{q_{12}=1}$ is equal to
\begin{equation} \label{eq:topofq33is1}
T_1^{2} T_2^{2}T_3^{3} (q_{13} - 1)^6  (q_{23} - 1)^6  q_{13}^{-3} q_{23}^{-3}.
\end{equation}
\end{Lemma}
\begin{proof}
It follows from the discussion in the previous proof that at a
generic point with $q_{12}=1$ the l.o.t.~of a theta constant vanishes if and
only if $(\e_1,\e_2,\e_3)=(1,1,0)$ and $(\de_1,\de_2)=(1,1)$ (with $\de_3$ arbitrary).
By the Fourier-Jacobi expansion \eqref{eq:fe_eps0}, these two theta constants have
a unique lowest order term of order $T_3^{1/2}=q_{33}^{1/2}$. Thus the evaluation of
the theta-null above has a l.o.t., and to compute it, it suffices
to collect the evaluation of the other factors of \eqref{eq:coeff444} at $q_{12}=1$
and the coefficient of the $q_{33}^{1/2}$-term in each of the two vanishing
theta constants to prove \eqref{eq:topofq33is1}.
\end{proof}

\begin{Rem} This lemma also admits a more geometric proof.
Indeed, the lemma amounts to showing that the generic vanishing order in $q_{33}$ of the
theta-null is higher for $\tau_{12}=0$ than for $\tau_{12}\ne 0$. To see this,
recall that the restriction of a theta constant
with $\eps_3=0$ to the locus where $\tau_{33}=i\infty$
is simply equal to the theta constant in dimension 2, with truncated characteristic. Thus in particular $$
 \tc{1&1&0}{1&1&\de}\left(\begin{smallmatrix} \tau_{11}&0&*\\ 0&\tau_{22}&*\\ *&*& i\infty\end{smallmatrix}\right)=
 \tc{1&1}{1&1}\left(\begin{smallmatrix} \tau_{11}&0\\ 0&\tau_{22} \end{smallmatrix}\right)=
 \tc{1}{1}(\tau_{11})\tc{1}{1}(\tau_{22})\equiv 0
$$
while
$$
 \tc{1&1&0}{1&1&a}\left(\begin{smallmatrix} *&*&*\\ *&*&*\\ *&*& i\infty\end{smallmatrix}\right)
 =\tc{1&1}{1&1}\left(\begin{smallmatrix} *&*\\ *&*\end{smallmatrix}\right)\not\equiv 0
$$
is not identically zero otherwise.
\end{Rem}

Finally, we thus get the main statement for this cone.
\begin{Prop}
There does not exist a non-complete Kuga curve $C\subset\Hyp_3$ not contained in $\Agdec[3]$, and  such that its closure $\overline{C}$ intersects the stratum corresponding to the cone $\sigma_{1+1+1}$.
\end{Prop}
\begin{proof}
By Proposition~\ref{pr:criclowestterm}  and Lemma~\ref{le:coeff444} the coefficient $C$ given in
\eqref{eq:coeff444} has to vanish.
Up to permutation of the variables, there are two cases depending on which factor vanishes.
If $q_{12} = 1$, then by Lemma~\ref{le:coeffq12is0} the term
$q_{11}^2 q_{22}^2 q_{33}^3$ divides all non-zero terms. Hence, by the same
argument as in Proposition~\ref{pr:criclowestterm}, its coefficient has
to vanish. This implies $q_{13} = 1$ or $q_{23} = 1$, and thus
the period matrix $\varphi(t)$ is diagonal, up
to the addition of an integral matrix, and by the action of $\mathop{Sp}_6(\ZZ)$ can
be conjugated into $\Agdec[3]$.

If one of the factors of $C$ with 4 summands vanishes, this is a relation in roots of unity.
If this relation is reducible, it has to be a sum of two relations with two terms. After permuting
the variables, it must then yield either the two relations
$\cq_{13}\cq_{23} =  1$ and $\cq_{13} =\cq_{23}$ or the two relations $\cq_{13}\cq_{23} =  -1$ and $\cq_{13}=-\cq_{23}$, so that the only possible solutions are $\pm1$. Thus we get $q_{13}=q_{23}=1$ (recall that $q_{ij}=\cq_{ij}^2$, so the signs disappear),
and $\varphi(t)$ is decomposable. If the relation with 4 summands is irreducible,
by Theorem~\ref{thm:mann} all the terms are roots of unity of degree at most 6.
Going through the possible cases easily shows that the only possible solution is still for the two corresponding $\cq_{ij}$ to be equal to $\pm 1$, so that we get again $C\subset\Agdec[3]$.
\end{proof}

\smallskip
We now deal with the other strata in $\barAVmoduli[3]$, case by case.
From Table~\ref{cap:listcones} we recall that there are two strata of codimension 4 and torus rank 3, and we now compute the l.o.t.~of $\theta_\nul$ in the corresponding coordinates.
\begin{Lemma} \label{le:c3.4coeff444}
For the cone $\sigma_{K_3+1}$, the modular form $\theta_\nul$ has a unique l.o.t.~in $T_1,\ldots,T_4$, given by
\begin{equation}
\label{eq:3.4acoeff444}
T_1^2 T_2^2 T_3^2 T_4^2 \cdot S_1^{-2} S_2^{-2} (S_2 - 1)^2   (S_1 - 1)^2  (S_1S_2 -1 )^2.
\end{equation}
For the cone~$\sigma_{C_4}$, the modular form $\theta_\nul$ has a unique l.o.t.~in $T_1,\ldots,T_4$ given by
\begin{equation}\label{eq:3.4bcoeff444}
 T_1^2 T_2^2 T_3^2 T_4^2\cdot (-\tS_1 -\tS_2 + 1)(-\tS_1 +\tS_2 - 1)(\tS_1 -\tS_2 - 1)(\tS_1 +\tS_2 + 1),
\end{equation}
where $\tS_j := S_j^{1/2}$.
\end{Lemma}
\begin{proof}
To determine the coefficient, if the unique l.o.t.~is known to exist, it of course suffices to look at the lowest term $C$ for the standard codimension 3 cone, given by \eqref{eq:coeff444}, and rewrite it in terms of the new unbounded variables.

For the cone $\sigma_{K_3+1}$, the unbounded variables are  $T_1=q_{11}q_{12},T_2=q_{22}q_{12},T_3=q_{33}$, and $T_4=q_{12}^{-1}$. We note first that for each term in the Fourier series for each theta constant the powers of the variables $T_1,T_2,T_3$ in this case are simply the powers of the diagonal entries $q_{11},q_{22},q_{33}$, i.e.~the same as powers of the variables $T_1,T_2,T_3$ for the cone $\sigma_{1+1+1}$. Thus it suffices to show that all other terms in the power series are divisible by at least $T_4^2$, to show that this term is the unique l.o.t.~for this expansion. Indeed, the power of $T_4=q_{12}^{-1}$ in a monomial in the Fourier series of an individual theta constant is easily computed to be
$$a_1 + a_2 - b_3 =  \frac12(n_1-n_2+\e_1/2-\e_2/2)^2$$
where $a_i$ and $b_i$ are given by \eqref{ab}. For $\e_1=\e_2$, the minimum value this attains is zero. For $\e_1\ne\e_2$, the minimal value of this expression is equal to $1/8$. Thus each theta constant has a l.o.t.~in $T_1,\ldots, T_4$, and the l.o.t.~of their product $\theta_\nul$ can be computed from reexpressing $C$ given by \eqref{eq:coeff444} in the new variable.

For the cone~$\sigma_{C_4}$, the situation is similar; we again have to show that the exponent of $T_4$ is at least two in every monomial. Here the exponent is given by
$$a_1+a_2+a_3+b_1+b_2+b_3 = \frac12(n_1+n_2+n_3+\e_1/2 +\e_2/2+\e_3/2)^2,$$
which attains a minimum of zero if and only if $\e_1+\e_2+\e_3\equiv 0\,(2)$, and
otherwise the minimum is equal to $1/8$. Thus each theta constant again has a unique l.o.t.,
and the l.o.t.~of $\theta_{\nul}$ is obtained from rewriting \eqref{eq:coeff444} in the new variables.
\end{proof}

The argument for the  other two cones follows the same pattern.
\begin{Lemma} \label{le:c3.5coeff444}
For the cone $\sigma_{K_4-1}$, the modular form $\theta_\nul$ has a unique l.o.t.~in $T_1,\ldots,T_5$, given by
\begin{equation}
\label{eq:c3.5coeff444}
T_1^2 T_2^2 T_3^2 T_4^2T_5^2\cdot(S_1 - 1)^2
\end{equation}
while for the cone~$\sigma_{K_4}$ the modular form $\theta_\nul$ has a unique l.o.t.~in $T_1,\ldots,T_6$, equal to
\begin{equation}
T_1^2 T_2^2 T_3^2 T_4^2T_5^2T_6^2.
\end{equation}
\end{Lemma}
\begin{proof}
Again, it suffices to show that a unique l.o.t.~exists, and then it can be obtained by
rewriting $C$ given by \eqref{eq:coeff444} in the corresponding variables. For $\sigma_{K_4-1}$, note that in each monomial the powers of $T_1,T_2,T_3,T_4$ are the same as for the cone $\sigma_{K_3+1}$, and thus it suffices to show that each monomial is divisible by $T_5$. This is done again by verifying that the corresponding power is a complete square. Similarly, for the cone $\sigma_{K_4}$ we now further have to check that each monomial is divisible by $T_6^2$, which is done in the same manner.
\end{proof}
We now use this existence of l.o.t.~to show that there do not exist Kuga curves whose closure intersect these strata.
\begin{Prop}
There does not exist a non-complete Kuga curve $C\subset\Hyp_3$ not contained in $\Agdec[3]$, and  such that its closure $\overline{C}$ intersects the stratum corresponding to one of the cones $\sigma_{K_3+1}$, $\sigma_{C_4}$, $\sigma_{K_4-1}$, $\sigma_{K_4}$.
\end{Prop}
\begin{proof}
By Proposition~\ref{pr:criclowestterm} and the previous lemmas, we first
focus on the vanishing of the l.o.t.~for each of these cones, but will then also need to consider the next lowest term, as above.

For the cone $\sigma_{K_3+1}$, if $S_1=1$ (of course $S_2=1$ is completely analogous), we claim that $\theta_\nul|_{S_1 = q_{13}=1}$ still has a unique l.o.t. Then this term must also be lowest order with respect to $q_{ii}$ for $i=1,2,3$, and must be the top coefficient
in \eqref{eq:topofq33is1} (replacing $q_{13}$ by $q_{12}$) with respect to the variable $q_{12}$,
which is $T_1^2T_2^3T_3^2T_4^2(S_2-1)^6S_2^{-3}$. But while proving Lemma~\ref{le:c3.4coeff444}
we have shown that the exponent of $T_4$ is always at least two, thus proving the existence of the l.o.t.~in this case also. Hence this next l.o.t.~has to vanish, and we conclude $q_{23}=1$, so that the period matrix is decomposable.

Still for the cone $\sigma_{K_3+1}$, if $S_1S_2=1$, then
applying the symplectic matrix $\left(\begin{matrix}
A & 0 \\ 0 & A^{-T} \\ \end{matrix}\right) \in \SpgZ$ with $A = \left(\begin{smallmatrix} 1&1&0\\
0&1&0\\ 0&0&1\end{smallmatrix}\right)$, we get $\tau' = A \tau A^T$. The matrix $A$ has been chosen so that
$$q'_{13} = \exp(2\pi i \tau'_{13}) = q_{13}q_{23} = S_1 S_2.$$
So we are now considering the expansion of $\theta_\nul|_{q'_{13}=1}$. Applying $A$ changes the
cone $\sigma_{K_3+1}$ to the cone with unbounded variables
$$T'_1= q'_{11} (q'_{12})^{-1}, T'_2= q'_{22} (q'_{12})^{-1}, T_3' = q'_{12}$$
and $T'_4 = q'_{33} $. We claim that there is a unique lowest order term in this
set of variables equal to the {\em lowest} order term in $q'_{12}$ of the expression
$$(q'_{12} - 1)^6  (q'_{23} - 1)^6  (q'_{12})^{-3} (q'_{23})^{-3}, $$
which has been obtained from \eqref{eq:topofq33is1} by replacing $q_{13}$ by $q_{12}$ and decorating
all variables by primes. Again, we only need to check that each monomial is divisible by $(T_4')^2$. The determination of the minimum in the proof in
Lemma~\ref{le:c3.4coeff444} is unchanged, when replacing $a_1 +a_2 -b_3$ by $a_1+a_2+b_3$,
and this concludes the
proof of the claim. It thus follows that  $q'_{23}=1$ and the Kuga curve is contained in $\Agdec[3]$.

For the cone $\sigma_{C_4}$, for the l.o.t.~to vanish we must have $\pm\tilde S_1\pm\tilde S_2+1=0$. Since this is a relation in roots of unity, by Mann's theorem it must have a solution with roots of order at most 6, and by going through the cases it then follows that $\tilde S_1$ and $\tilde S_2$ are sixth roots of unity, and we again have $S_1S_2=1$ (with both of them being third roots of unity). The four factors of \eqref{eq:3.4bcoeff444} correspond to the l.o.t.~of theta constants with $(\e_1,\e_2,\e_3) = (1,1,1)$ and varying $\delta$, so that for these $S_1,S_2$ one of these factors vanishes. This term corresponds
to a choice of $n_i$ where $a_1+a_2+a_3+b_1+b_2+b_3$ is minimal, equal to $1/8$ with the given
$\e_i$. But then there exist $n_i$ where the forms $a_i$ still attain their minima
while this expression takes the next smallest value, i.e.~$a_1+a_2+a_3+b_1+b_2+b_3 = 9/8$. The two corresponding choices
are precisely $(0,0,0)$ and $(-1,-1,-1)$, but then the corresponding
coefficient is the same for both choices, in fact equal (up to a factor of a root of unity) to $S_1^{3/8}S_2^{3/8}$ for each such theta
characteristic. Since $S_1S_2 = 0$ contradicts $S_i$ being a root of
unity, this implies that there are no Kuga curves whose closure intersects this
boundary stratum.

For the cone $\sigma_{K_4-1}$, the l.o.t.~can only vanish for $S_1=1$. In this case
we again have to look at the next lowest order term after substituting
$S_1 = q_{23} = 1$. The candidate for this term is the term of top
degree in $q_{13}$ and $q_{12}$ in \eqref{eq:topofq33is1}, with the roles of $q_{23}$ and
$q_{12}$ exchanged. That this is indeed the l.o.t.~follows with the same
minimization argument as in the previous cases. But this term is just a
non-zero constant and consequently, there does not exist a Kuga curve whose closure intersects this
boundary stratum.

For the cone $\sigma_{K_4}$ the lowest order term is manifestly non-zero as each $T_i\in\CC^*$, so the closure of $\theta_\nul$
does not even intersect this boundary stratum.
\end{proof}

Since we have now gone through all the possible strata of torus rank~3,
it thus follows that there do not exist non-complete Kuga curves $C\subset\Hyp_3$ such that their closure intersects the locus of semiabelic varieties of torus rank 3.

\section{Examples of Shimura curves of torus rank 2} \label{torusrank2}
For the case of torus rank two, we have the strata corresponding to the two cones $\sigma_{1+1}$ and $\sigma_{K_3}$, and in the Fourier-Jacobi expansion the moduli coordinates will enter. We will thus attempt to proceed as above in ruling out the existence of hyperelliptic Kuga families of torus rank 1 and 3, but --- which was very surprising for us --- it will turn out along the way that such examples actually exist. While at the moment we are unable to fully classify examples of such Kuga curves (see Remark~\ref{rem:notdone}), we will naturally end up constructing an infinite collection of {\em Shimura} (as opposed to just Kuga) curves given by Theorem \ref{thm:construct}, and finishing the proof of Theorem \ref{thm:main} for all the other cases.

We will first deal with the cone $\sigma_{1+1}$ in Table~\ref{cap:listcones}, discovering our examples, and then show that none exist for $\sigma_{K_3}$. For the case of $\sigma_{1+1}$ the Fourier-Jacobi expansion can be computed by applying formulas \eqref{eq:fe_eps0} and \eqref{eq:fe_eps1} twice. We will work with the period matrix $\tau$, and denote for brevity $\tau':=\left(\begin{smallmatrix} \tau_{22}&\tau_{23}\\\tau_{23}&\tau_{33}\end{smallmatrix} \right)$ its lower block.
The Fourier-Jacobi expansion of the theta function is then given by
\begin{flalign}
&\thetat{0& 0& \e_3}{\de_1& \de_2& \de_3} ( \tau,0)=& \label{eq:thetaexp00}\\
& =\thetat{0&\e_3}{\de_2&\de_3}\left(\tau',0\right)
 \,\,+\,\,2 q_{11}^{1/2} (-1)^{\de_1} \thetat{0&\e_3}{\de_2&\de_3}\left(\tau',
 \ \begin{matrix}\tau_{12}\\ \tau_{13}\end{matrix}\right) &\nonumber\\
&=\theta_m(0)\,\,+\,\,2q_{22}^{1/2}(-1)^{\de_2}\theta_m(\tau_{23}) \,\,+\,\,2q_{11}^{1/2}(-1)^{\de_1}  \theta_m(\tau_{13})&\nonumber\\
& \qquad\ \quad\ \,\, + \,\, 2 q_{11}^{1/2}q_{22}^{1/2}(-1)^{\de_1+\de_2} \left(q_{12}\theta_m(\tau_{13}-\tau_{23})+q_{12}^{-1} \theta_m(\tau_{13}+\tau_{23})\right) & \nonumber \\
&  \qquad\  \quad\ \,\, + \,\,O(q_{11}^{3/2}, q_{22}^{3/2})  & \nonumber
\end{flalign}
up to higher order terms, where for brevity we have denoted $\theta_m(a):=\thetalc{\e_3}{\de_3}(\tau_{33},a)$ the corresponding theta function with characteristic on the elliptic curve $E:=\CC/(\ZZ+\ZZ\tau_{33})$.

Similarly, for the case when  $\e_1=1, \e_2=0$  we get
\begin{flalign}
& \thetat{1& 0& \e_3}{\de_1& \de_2& \de_3} ( \tau,0) = 2 i^{\de_1}
 q_{11}^{1/8}\thetat{0&\e_3}{\de_2&\de_3}\left(\tau',\ \begin{matrix}\tfrac{\tau_{12}}{2} \\
\tfrac{\tau_{13}}{2}\end{matrix}\right) & \\
& =2 i^{\de_1}\,q_{11}^{1/8}\theta_m(\tfrac{\tau_{13}}{2})\,\,+\,\,2 i^{\de_i + 2\de_2} q_{11}^{1/8}q_{22}^{1/2}\,\,\cdot & \nonumber \\
& \qquad
\cdot \left(q_{12}^{1/2}\theta_m(\tfrac{\tau_{13}}{2}+\tau_{23})+q_{12}^{-1/2}
\theta_m(\tfrac{\tau_{13}}{2}-\tau_{23})\right) \,+\, O(q_{11}^{9/8}, q_{22}^{3/2}). &\nonumber
\end{flalign}
The case $\e_1=0, \e_2=1$ is of course completely analogous, and we do not need to treat it separately.
For the case $\e_1=\e_2=1$ it will turn out that we will need the full Fourier-Jacobi expansion, which we now give, along with its first few terms:
\begin{flalign}
\label{FJ11}
&\thetat{1& 1& \e_3}{\de_1& \de_2& \de_3} ( \tau,0) = & \\
 &=\sum_{N_1\in\ZZ}q_{11}^{(2N_1+1)^2/8}i^{(2N_1+1)\de_1}\thetat{1&\e_3}{\de_2&\de_3}\left(\tau',
\begin{matrix}(2N_1+1)\tau_{12}/2\\ (2N_1+1)\tau_{13}/2\end{matrix}\right)& \nonumber\\
&=\sum_{N_1,N_2\in\ZZ}q_{11}^{n_1^2/8}q_{22}^{n_2^2/8}q_{12}^{n_1n_2/4}
i^{n_1\de_1+n_2\de_2} \,
\thetalc{\e_3}{\de_3}\Bigl(\tau_{33},\frac{n_1\tau_{13}+n_2\tau_{23}}{2}\Bigr) & \nonumber\\
&=2\cq_{11}\cq_{22}\left(
 \cq_{12}i^{\de_1+\de_2}\,\theta_m(\tfrac{\tau_{13}+\tau_{23}}{2})\,\,+\,
  \cq_{12}^{-1}i^{\de_1-\de_2}\,\theta_m(\tfrac{\tau_{13}-\tau_{23}}{2})\right) & \nonumber\\
&+2\cq_{11}\cq_{22}^{9}\left(\cq_{12}^{3}
 i^{\de_1+3\de_2}\,\theta_m(\tfrac{\tau_{13}+3\tau_{23}}{2})\,\,+\,
  \cq_{12}^{-3}i^{\de_1-3\de_2}\,\theta_m(\tfrac{\tau_{13}-3\tau_{23}}{2})\right)& \nonumber\\
&+2\cq_{11}^{9}\cq_{22}\left(\cq_{12}^{3}
 i^{3\de_1+\de_2}\,\theta_m(\tfrac{3\tau_{13}+\tau_{23}}{2})\,\,+\,
  \cq_{12}^{-3}i^{3\de_1-\de_2}\,\theta_m(\tfrac{3\tau_{13}-\tau_{23})}{2})\right)& \nonumber\\
&+2\cq_{11}^{\,9}\cq_{22}^{\,9}\left(\cq_{12}^{\,9}
 i^{3\de_1+3\de_2}\,\theta_m(\tfrac{3\tau_{13}+3\tau_{23}}{2})\,+\,
  \cq_{12}^{\,-9}i^{3\de_1-3\de_2}\,\theta_m(\tfrac{3\tau_{13}-3\tau_{23}}{2})\right)& \nonumber \\
&+ O(\cq_{11}^{\,16},\cq_{22}^{\,16}) & \nonumber
\end{flalign}
where for convenience we have denoted
\begin{equation}\label{conv}
n_i:=2N_i+1;\quad \cq_{11}:=q_{11}^{1/8};\quad \cq_{22}:=q_{22}^{1/8};\quad
\cq_{12}:=q_{12}^{1/4},
\end{equation}
and the 2's in the explicit expression appear, since the terms for $(n_1,n_2)$ and $(-n_1,-n_2)$ are equal.

The expansions in these 3 cases are very different, and it turns out will have very different properties, so we will treat them separately

\begin{Prop} \label{prop00}
A theta constant with $\eps_1=\eps_2=0$ cannot vanish identically on a non-complete Shimura curve $C\subset Hyp_3$ (in the lifting to $\HH_3$ given by Lemma \ref{le:V2ndZmatrixform}), whose closure intersects the boundary stratum corresponding to $\sigma_{1+1}$.
\end{Prop}
\begin{proof}
Indeed, in this case by \eqref{eq:thetaexp00} the constant term of the theta constant is then $\theta_m(0)$, which for finite $\tau_{33}$ could only vanish if $\e_3=\de_3=1$. But in this case $[\e,\de]$ is an odd characteristic, while by definition the theta-null is the product of all theta constants with even characteristics.
\end{proof}

\begin{Prop} \label{prop10}
A theta constant with $\eps_1=1,\eps_2=0$ cannot vanish identically on a non-complete Shimura curve $C\subset Hyp_3$ (in the lifting to $\HH_3$ given by Lemma \ref{le:V2ndZmatrixform}), whose closure intersects the boundary stratum corresponding to $\sigma_{1+1}$.
\end{Prop}
\begin{proof}
As the theta constant with $\eps_1=0$ has no constant term, we first need to ascertain that the corresponding power series (with exponents in $(1/8)\ZZ$) in $T_1=q_{11}$ and $T_2=q_{22}$ has a unique lowest order term. This is of course clear from the formulas, as for $\e_1=1,\e_2=0$ all the terms are divisible by $q_{11}^{1/8}$, while for $\e_1=\e_2=1$ they are all divisible by $q_{11}^{1/8}q_{22}^{1/8}$.

Thus the lowest order term is then $q_{11}^{1/8}\theta_m(\tfrac{\tau_{13}}{2})$. For this term to vanish
we need that $\tau_{13}/2=(\tau_{33}(\e_3+1)+\de_3+1)/2$ is the 2-torsion point on $E$ determined by the characteristic.
Consequently, $\tau_{13}=0$, as a point on $E$. Then, since $\theta_m$ is an odd function around its zero $\tfrac{\tau_{13}}{2}$, we have $\theta_m(\tfrac{\tau_{13}}{2}+\tau_{23})=-\theta_m(\tfrac{\tau_{13}}{2}-\tau_{23})$, and for the next lowest term we get
$$
  q_{12}^{1/2}\theta_m(\tfrac{\tau_{13}}{2}+\tau_{23})+q_{12}^{-1/2}\theta_m(\tfrac{\tau_{13}}{2}-\tau_{23})=
  (q_{12}^{1/2}-q_{12}^{-1/2})\theta_m(\tfrac{\tau_{13}}{2}+\tau_{23})
$$
This can vanish only if $q_{12}=1$, in which case $\tau_{12}=0$, and the period matrix is decomposable, or if $\theta_m(\tfrac{\tau_{13}}{2}+\tau_{23})=0$, in which case since also $\theta_m(\tfrac{\tau_{13}}{2})=0$, and theta function in one variable has only one zero, we must have $\tau_{23}=0$, and the period matrix is again decomposable.
\end{proof}

\noindent {\bf Case $\e_1=\e_2=1$.}

The situation is strikingly more difficult here, and we discover the examples of Shimura curves. Indeed, recall that our period matrix $\tau$ is in fact given by \eqref{eq:varphinew}, and we are supposed to order the terms in the expansion by their power of $t$. In this case, the values of rational numbers $e_{11}$ and $e_{22}$ in \eqref{eq:varphinew}  matter for ordering the terms, as we have $q_{11}=e^{2\pi i b_{11}}q^{e_{11}}$ and $q_{22}=e^{2\pi i b_{22}}q^{e_{22}}$ (where we recall $q=e^{2\pi i t}$). Still, for the lowest order term there is no case distinction: it comes from $\cq_{11}\cq_{22}$, and for it to vanish is equivalent to
$$
 \cq_{12}^2\theta_m(\tfrac{\tau_{13}+\tau_{23}}{2})+
 (-1)^{\de_2}\theta_m(\tfrac{\tau_{13}-\tau_{23}}{2})=0.
$$
If the coefficient $\theta_m(\tfrac{\tau_{13}+\tau_{23}}{2})$ of $\cq_{12}$
vanishes, then this equation implies that also $\theta_m(\tfrac{\tau_{13}-\tau_{23}}{2})=0$.
Since the theta  function has a unique zero up to a shift by lattice points,
we deduce in this case that $\tau_{23}$ must be zero. Then furthermore $\tfrac{\tau_{13}}{2}$ must be the two-torsion point where the theta function vanishes, and thus finally $\tau_{13}=0$ and the period matrix is decomposable.

Consequently, independently of the relative size of $e_{11}$ and $e_{22}$ we must have
\begin{equation}\label{term1}
 \cq_{12}^2=(-1)^{\de_2+1}\frac{\theta_m(\tfrac{\tau_{13}-\tau_{23}}{2})}{\theta_m(\tfrac{\tau_{13}+\tau_{23}}{2})}.
\end{equation}

However, for the next lowest order term we have three possibilities, depending on the valuations.
\begin{Prop}
A theta constant with $\eps_1=\eps_2=1$ cannot vanish identically on a non-complete Shimura curve $C\subset Hyp_3$ , whose closure intersects the boundary stratum corresponding to $\sigma_{1+1}$ unless we also have $e_{11}=e_{22}$ in the lifting to $\HH_3$ given by Lemma \ref{le:V2ndZmatrixform}.
\end{Prop}
\begin{proof}
Assume without loss of generality that $e_{11}<e_{22}$. Then the next lowest order term in the Fourier-Jacobi expansion is equal to $2\, i^{3\de_1-\de_2}$ times
\begin{equation}\label{term2case1}
\cq_{11}^{9}\cq_{22}\cdot \left(\cq_{12}^{3}(-1)^{\de_2}
 \,\theta_m(\tfrac{3\tau_{13}+\tau_{23}}{2})+
  \cq_{12}^{-3}\,\theta_m(\tfrac{3\tau_{13}-\tau_{23}}{2})\right)
\end{equation}
Combining this equation with \eqref{term1} and eliminating $\cq_{12}$ from both equations we get the identity
$$
 \theta_m(\tfrac{3\tau_{13}+\tau_{23}}{2})\,\theta_m^3(\tfrac{\tau_{13}-\tau_{23}}{2})= \theta_m(\tfrac{\tau_{13}-\tau_{23}}{2})\,\theta_m^3(\tfrac{\tau_{13}+\tau_{23}}{2}).
$$
Denoting $u:=\tau_{23}/2$ and $v:=\tau_{13}/2$, we thus consider this equality as a function of $v$: it is the identity
$$
 \theta_m(3v+u)\,\theta_m^3(v-u)= \theta_m(3v-u)\,\theta_m^3(v+u),
$$
which as a function of $u$ is a section of the bundle $4\Theta$ on $E$, and thus for fixed $v$ must have 4 zeroes in $u$, counted with multiplicity.
Notice however that if we take $u=n$ to be any two-torsion point on the elliptic curve, then this identity becomes simply
$$
 \theta_{m+n}(3v)\,\theta_{m+n}^3(-v)= \theta_{m+n}(-3v)\,\theta_{m+n}^3(v)
$$
which is true for any 2-torsion point $n$ (the signs match irrespective of the parity of $m+n$). Thus we have found all the four roots of this equation, and thus the only
solutions are for $u=\tau_{23}/2$ to be a 2-torsion point, which means that $\tau_{23}=0$. Similarly to the case $\e_1=\e_2=0$ above, from \eqref{term1} we then get $q_{12}=1$, so that $\tau_{12}=0$, and the period matrix again becomes decomposable.
\end{proof}

\noindent {\bf Subcase $e_{11}=e_{22}$:}

Finally, this is the only possible situation when we get the examples. By reparameterizing $t$ we can then
assume $\cq_{11}=q$, and denote $\gamma:=\cq_{22}/\cq_{11}=e^{\pi i(b_{22}-b_{11})/4}$. As in this case our goal turns out to be constructing infinitely many examples, for simplicity of computation let us take $\delta_1=\delta_2=1$ and $\eps_3=\delta_3=0$, in which case we will already exhibit infinitely many distinct Shimura curves.
For brevity, we denote
$$
  \theta(n_1,n_2):=\thetalc{0}{0}\Bigl(\tau_{33},\frac{n_1\tau_{13}+n_2\tau_{23}}{2}\Bigr).
$$
Next, for further use we rewrite the Fourier-Jacobi series given by \eqref{FJ11}, by gathering together the terms for both signs of $n_2$, and also for interchanging $n_1$ and $n_2$. We obtain
\begin{flalign}
\frac{1}{2}\,&\thetat{1& 1& 0}{1& 1&0} ( \tau,0) = & \label{FJfinal}\\
&= \sum_{{n_1>0{\rm\ odd,\ \ }} \atop {n_2{\rm\ odd}}} \!\!\!\!\cq_{11}^{n_1^2}\,\,\cq_{22}^{n_2^2}\,\cq_{12}^{n_1n_2}
i^{n_1+n_2} \,\theta(n_1,n_2) & \nonumber\\
& = \sum_{{{n_1 > n_2> 0,}\atop{\,{\rm both\,  odd}}}} \!\!\!\!q^{n_1^2+n_2^2}\,\cq_{12}^{n_1 n_2} \Bigl(i^{n_1+n_2}\,
\theta(n_1,n_2)
+i^{n_1-n_2}\, \cq_{12}^{-2n_1 n_2}\, \theta(n_1,-n_2)\Bigr. &\nonumber \\
&\phantom{=}  + \Bigl. i^{n_2+n_1}\,\gamma^{n_1^2-n_2^2}\,\theta(n_2,n_1)
+i^{n_2-n_1}\,\cq_{12}^{-2n_1 n_2}\,\gamma^{n_1^2-n_2^2}\,\theta(n_2,-n_1)\Bigr) &\nonumber \\
&\phantom{=} + \sum_{n_1 >0 \, \,{\rm odd}} q^{2n_1^2}\,\,(-1)^{n_1}\,\,\cq_{12}^{n_1^2} \left(\theta(n_1,n_1)-\cq_{12}^{-2n_1^2}\theta(n_1,-n_1)\right), & \nonumber
\end{flalign}
where the last sum accounts for the terms with $n_1=n_2$.

The simplest way in which this power series can vanish is if every summand vanishes, that is the sum of two terms in the last line is zero for any $n_1$, and the sum of the four terms in the two preceding lines is zero for any $n_1>n_2$. The vanishing of the sum of two terms in the last line can most easily be achieved
if the values of the theta function $\theta(n_1,n_1)$ and $\theta(n_1,-n_1)$ are related by the automorphy property. By inspection, this would be the case if we choose $\tau_{33}=i$, and require $\tau_{13}=i\tau_{23}$. Indeed, using $\left(\begin{smallmatrix} 0&1\\ -1& 0\end{smallmatrix}\right)i=i$ the theta transformation formula in genus one reads (where we drop the argument $\tau=i=\tau_{33}$ of the theta function
\begin{equation}\label{i}
 \thetalc{\de_3}{\e_3}(iv)= i^{-\e_3\de_3}\exp(\pi v^2)\thetalc{e_3}{\de_3}(v).
\end{equation}

By Corollary~\ref{cor:AZAterm} we know that $\tau_{13} = r_{13}+ai$, $\tau_{23} = r_{23}+bi$
and $\tau_{12} = r_{12}+ab\,i$ for $a,b,r_{ij} \in \QQ$, and the condition $\tau_{13}=i\tau_{23}$
thus gives $\tau_{13} = -b+ai$ and $\tau_{23} = a+bi$.

Thus using \eqref{i}, we get from \eqref{term1}
\begin{equation} \label{eq:q12fordimV1}
\cq_{12}^2 = \exp(\pi i\tau_{23}^2/2).
\end{equation}
This holds if $r_{12}+abi\equiv i\tau_{23}^2/2 \mod \ZZ$,
i.e.~if $r_{12} \equiv (a^2-b^2)/2 \mod \ZZ$.

We now consider the next term in the Fourier-Jacobi expansion, corresponding to $(n_1,n_2)=(3,1)$ in \eqref{FJfinal}, where we get the following 4 summands:
\begin{equation}\label{defininggamma}
\begin{aligned}
-  \gamma^8\, & \Bigl(\cq_{12}^{-6}  \thetalc{0}{0}(\tfrac{1+3i}{2}\tau_{13}) \,-\, \thetalc{0}{0}(\tfrac{1-3i}{2}\tau_{13}) \Bigr) \\
  & +\, \thetalc{0}{0}(\tfrac{3-i}{2}\tau_{13}) \,-\, \cq_{12}^{-6} \,\thetalc{0}{0}(\tfrac{3+i}{2}\tau_{13}) \,= \,0,
\end{aligned}
\end{equation}
We notice that the ratio of the first and the third summand is the same as the ratio of the second
and fourth summand, as follows from the value of $\cq_{12}$ given
by \eqref{term1}, and the theta transformation formula \eqref{i}.
We can thus solve to obtain
\begin{equation} \label{eq:gammafordimV1}
 \gamma = \exp(-\pi \tau_{13}^2/4) = \exp(\pi \tau_{23}^2/4).
\end{equation}
We thus conclude that we must have $r_{22} \equiv 2ab \mod \ZZ$. Finally, substituting all of these values in Corollary \ref{cor:AZAterm}, this means that the lift of the Shimura curve to $\HH_3$ given by Corollary \ref{cor:AZAterm} has the form
\begin{equation}\label{examples}
 \varphi(t)=\left(\begin{array}{rrr} t+a^2i&\tfrac{a^2-b^2}{2}+abi&-b+ai \\ \tfrac{a^2-b^2}{2}+abi&
 t+2ab+b^2i &a+bi\\ -b+ai&a+bi& i \end{array}\right)
\end{equation}
By Remark~\ref{rem:shimura_crit} this expression indeed defines a Kuga curve.
With $u=a+bi\in\QQ+\QQ i$, we notice that this is precisely the same expression as \eqref{gen3examples}
up to the shift of the parameter $t \mapsto t+2ab+b^2i$.

Moreover, since $\tau_{33}=i$ defines an elliptic curve with complex multiplication, for $t=i$
the abelian variety defined by $\varphi(t)$ is isogenous to $E_i^3$. Thus the family has a CM point, and defines a Shimura curve (not just Kuga).
To prove Theorem~\ref{thm:construct} it suffices to show that the Shimura curves constructed above
lie in the hyperelliptic locus.

\begin{proof}[Proof of Theorem~\ref{thm:construct}]
Indeed, to show that the Shimura curves given by the above equation or \eqref{gen3examples} lie in the hyperelliptic locus, it suffices to show that some theta constant vanishes identically on them. The vanishing theta constant is of course the one with characteristic $\left[\begin{smallmatrix}1&1&0\\ 1&1&0\end{smallmatrix}\right]$, and to show that it vanishes we need to show that its Fourier-Jacobi expansion given by \eqref{FJfinal} vanishes term by term. The vanishing of the sum of two terms in the last line, for any $n_1$, follows immediately upon substituting the value of $\cq_{12}$ given by \eqref{eq:q12fordimV1} and using the theta transformation formula \eqref{i}.

We now consider the sum of four terms in \eqref{FJfinal} for $n_1>n_2$. We divide the expression by $i^{n_1+n_2}$, use
that both $n_i$ are odd and  substitute the value of $\cq_{12}$ from \eqref{eq:q12fordimV1} and the
value of $\gamma = \exp(-\pi \tau_{13}^2/4)$ from \eqref{eq:gammafordimV1} to obtain
$$\begin{aligned}
\theta&(n_1,n_2) -  \cq_{12}^{-2n_1 n_2}\, \theta(n_1,-n_2)
+ \gamma^{n_1^2-n_2^2}\, \left(\theta(n_2,n_1)
- \,\cq_{12}^{-2n_1 n_2}\,\,\theta(n_2,-n_1) \right)\\
=& \,\thetalc{0}{0}(\tfrac{n_1-n_2i}{2}\tau_{13})
-\exp(\pi i \tfrac{n_1n_2}{2}\tau_{13}^2)\,\thetalc{0}{0}(\tfrac{n_1+n_2i}{2}\tau_{13}) \\
   & + \exp(-\pi \tfrac{n_1^2-n_2^2}{4}\tau_{13}^2)\,\left(\thetalc{0}{0}(\tfrac{n_2-n_1i}{2}\tau_{13})
   - \exp(\pi i\tfrac{n_1n_2}{2}\tau_{13}^2)\,\thetalc{0}{0}(\tfrac{n_2+n_1i}{2}\tau_{13})\right)\\
\end{aligned}
$$

Using formula \eqref{i} we obtain
$$
  \thetalc{0}{0}(\tfrac{n_2-n_1i}{2}\tau_{13})=\exp(\pi\tfrac{(n_1+n_2i)^2}{4}\tau_{13}^2)\,\thetalc{0}{0}(\tfrac{n_1+n_2i}{2}\tau_{13})
$$
and similarly
$$
  \thetalc{0}{0}(\tfrac{n_2+n_1i}{2}\tau_{13})=\exp(\pi\tfrac{(n_1-n_2i)^2}{4}\tau_{13}^2)\,\thetalc{0}{0}(\tfrac{n_1-n_2i}{2}\tau_{13}),
$$
so that the exponential factors  involving $\tau_{13}^2$  in the formula
above match pairwise.
Thus the Fourier-Jacobi expansion in this case vanishes identically in $t$, as it vanishes term by term.

To see that the curve given by \eqref{examples} for any $a,b \in \QQ$,
not both integers, is not contained in $\cA_3^{dec}$,
it suffices to check that no other theta constant vanishes identically in $t$ --- since a smooth hyperelliptic Jacobian of genus 3 only has one vanishing theta constant. By Propositions~\ref{prop00}
and~\ref{prop10} this could happen only for $\e_1=\e_2=1$. If $\e_3=\de_3$, so that on
both sides of the theta transformation formula the same characteristic appears, we combine
\eqref{term1} and \eqref{i} to obtain
$$ \cq^2_{12} = (-1)^{\de_2+1} \,\,i^{-\e_3^2}\, \exp(\pi i \tau_{23}^2/2).$$
Comparing with \eqref{eq:q12fordimV1}, we conclude $\de_2=1$ and  $\e_3=\de_3=1$, hence $\de_1=1$ by the
parity.  If $\e_3 \neq \de_3$ we obtain from comparing \eqref{eq:q12fordimV1} to
\eqref{term1} and \eqref{i} that
$$\thetalc{1}{0}(\tfrac{1+i}{2}\tau_{23}) = (-1)^{\delta_2 + 1}\thetalc{0}{1}(\tfrac{1+i}{2}\tau_{23}).$$
The ratio $f(z) = \left(\thetalc{1}{0}(z)/\thetalc{0}{1}(z)\right)^2$ is the  ratio of two sections of the line bundle $2\Theta$ on the elliptic curve $E_i=\CC/\ZZ+\ZZ i$ given by $\tau_{33}=i$, and since the degree of $\Theta$ is equal to one, it follows that $f: E_i \to \PP^1$ gives a degree two map.
Since both the numerator and denominator are even functions, $f$ is an even map, so that $f(z)=f(-z)$, and the branch points of the cover are thus at $0, \tfrac12, \tfrac{i}2$, and $\tfrac{1+i}2$. By the transformation formula we have $f(0) = 1$. Since $0$ is a branch point
there cannot be any other point $z\neq 0$ such that $f(z)=1$, and thus we must have $\tau_{23}=0$, which finally proves that the curves constructed above are not contained in $\cA_3^{dec}$.

To prove the theorem it thus remains to show that \eqref{examples} for $a,b$ varying
in $\QQ$ yields infinitely many {\em distinct}  Shimura curves in $\Hyp_3$, which is the
content of the following proposition.
\end{proof}
\begin{Prop} \label{prop:infdisjoint}
Let $C_u$ be the Shimura curve
defined by the period matrix \eqref{gen3examples}. As $u$ varies in $\QQ + i\QQ$
the degree of the restriction of the polarization to the fixed part is unbounded.

In particular, there are infinitely many different Shimura curves
among the $C_u$.
\end{Prop}
\begin{proof}
The fixed part of Shimura curves given by \eqref{gen3examples}  is
$V / \Lambda$, where $V = \CC (0,0,1)^t$ and $\Lambda$ is
the intersection of the column span of the period matrix $\Pi_u(t)$ with $V$. In
the special case $u = n^{-1} + n^{-1} \,i$ for $n \in \ZZ$, one calculates  that
$$\Lambda = \langle n^2c_1+n^2c_2 - (2n^2-1)c_3 -n^2(2n^2-1)c_4 + n^2c_5 \, , \, c_6
\rangle,$$
where $c_i$ are the columns
of the period matrix $\Pi_u(t)$. This implies that
$\Lambda = \langle (0,0,n^3i)^T,(0,0,1)^T \rangle$.
Hence the degree of the polarization restricted to the fixed part of the
Shimura curve is equal to $n^3$ (and is thus unbounded as $n$ grows). Since this degree is an intrinsic property of the
Shimura curve, independent of the period matrix representative, the curves $C_u$ with $u=(1+i)/n$ are all distinct, proving the second claim.
\end{proof}
\begin{Rem} We remark that
the two-dimensional subvariety of $\HH_3$ given by \eqref{gen3examples}
parameterized by $u,t$ is not a two-dimensional
Shimura variety (maybe not even an algebraic subvariety of $\cA_3)$. In fact, if it were a Shimura variety,
the unipotent part of the stabilizer of $\infty$ in $\Sp_6(\ZZ)$ would have to intersect the stabilizer of this this two-dimensional object
in a group of rank two. A unipotent element stabilizing $\infty$  acts by translating the period matrix by $B = (b_{ij}) \in \ZZ^{3 \time 3}$. If $\Pi_u(t)+B = \Pi_{u'}(t')$, then comparison of the $(1,3)$ and $(2,3)$ entries implies
$u=u'$ and $b_{13} = b_{23}$. We first deduce that $t' = t + b_{11}$, $b_{11} = b_{22}$ and that all
other entries of $B$ are zero, hence that the intersection has at most rank one. In Section~\ref{sec:family} we will show that the curve obtained above for $a=b=1/2$ in \eqref{examples} indeed gives a Shimura curve corresponding to the locus with a given automorphism group.
\end{Rem}

The proof of the last case is analogous to the previous propositions,
but significantly easier --- and leads to no further examples.

\begin{Prop}
There does not exist a non-complete Kuga curves $S\subset\Hyp_3$ whose closure in $\barAVmoduli[3]$ contains a point in the stratum corresponding to $\sigma_{K_3}$.
\end{Prop}
\begin{proof}
For the case $\e_1=\e_2=0$ there is simply no change in the constant term, and the same argument as for $\sigma_{1+1}$ applies. For $\e_1=1\ne 0=\e_2$ the lowest order term is still the same, while the next term only contains one summand (recall that we have $q_{11}=T_1T_3$, $q_{22}=T_2T_3$, and are now expanding also in $q_{12}=T_3^{-1}$), and is now equal to   $q_{11}^{1/8}q_{12}^{1/2}\theta_m(\tfrac{\tau_{13}}{2}+\tau_{23})$,
which again can only vanish together with the lowest order term if $\tau_{13}=\tau_{23}=0$, so that the period matrix is decomposable.

Finally for the case $\e_1=\e_2=1$, in the lowest order term that we had for $\sigma_{1+1}$ we only have one summand of lowest order in $q_{12}^{-1}=T_3$, and thus we must have
$$
 \theta_m(\tfrac{\tau_{13}+\tau_{23}}{2})=0,
$$
which means that this must be the value of the theta constant at a 2-torsion point of the elliptic curve $E$, and thus we must have $\tau_{13}+\tau_{23}=0$ on $E$. To determine the next lowest term, we note that $q_{12}^{-1/2}=T_3^{1/2}$, while $q_{12}^{1/2}q_{22}=T_2T_3^{1/2}$. Thus the next lowest order term must be simply the next summand of what was the lowest order term for the cone $\sigma_{1+1}$, i.e.~$(-1)^{\de_2}T_3^{1/2}\theta_m(\tfrac{\tau_{13}-\tau_{23}}{2})$. For this to also vanish implies $\tau_{13}-\tau_{23}=0$, so that we finally have $\tau_{13}=\tau_{23}=0$, and the period matrix is again decomposable.
\end{proof}
\begin{Rem}\label{rem:notdone}
It would be interesting to enumerate {\em all} non-complete Shimura curves
contained in $\Hyp_3$; as the above (together with the proposition below) shows that the closure of these must all intersect the boundary stratum corresponding to $\sigma_{1+1}$, must have $e_{11}=e_{22}$, and the vanishing theta constant must have $\eps_1=\eps_2=1$. For $\de_1=\de_2=0=\eps_3=\de_3$ and
$\tau_{33}=i$ one can construct similarly to the above an infinite number of such Shimura curves
by setting each term in the  version of \eqref{FJfinal} for the other characteristic
equal to zero. More generally, the same construction still works if one replaces
$\tau_{33}=i$ by $\tau_{33} \in \QQ+\QQ i$. So a first step is to generalize
Proposition~\ref{prop:infdisjoint} to a classification of the resulting curves up to
the action of $\Sp_6(\ZZ)$.

Beyond that, the first step towards a complete classification would be to solve~\eqref{term1} in complete
generality, that is to determine for which $\tau_{33} \in \HH$ (besides $\tau_{33} \in
\QQ+\QQ i$) there exist rational numbers $a,b, r_{12},r_{13},r_{23}$, such
that
$$\exp(2\pi i (r_{12}+ ab\tau_{33}))
= (-1)^{\delta_2 + 1} \frac{\theta_m(r_{13}-r_{23} +  (a-b) \tau_{33})}
{\theta_m(r_{13} + r_{23} +  (a+b) \tau_{33})}\,\, .$$
\end{Rem}
\section{The $(\ZZ_2 \times \ZZ_4)$-family} \label{sec:family}
In this section we consider the family of hyperelliptic curves of genus three that are obtained as cyclic covers of $\PP^1$ such that generically the reduced (modulo the hyperelliptic involution) automorphism group is equal to $\ZZ_2\times\ZZ_4$. This family appears in the table of Moonen and Oort in \cite[Table~2 in \S~6, Eq.~(22)]{moor} and was recently investigated in detail by Lu and Zuo in \cite{luzuo} who verify that it is a non-complete Shimura curve. Besides determining singular
fibers, they also checked that the family, given as an abelian cover in \cite{moor}
has explicitly the hyperelliptic equation
$$X_s: w^2 = z(z+1)(z-1)(1+2sz^2+z^4),$$
where $s\in\CC$ is a parameter.
This family admits two automorphisms of order
four given by $T_0(z,w) = (-z,i w)$ and $T_1(z,w) = (1/z, i\, w/z^4)$. Since by definition $T_0^2=T_1^2$ is the hyperelliptic involution, for any $s$ the reduced
automorphism group of $X_s$ contains (and by \cite[Equation~(10)]{KuKo} in fact
equals) $\ZZ_2 \times \ZZ_4$.
We will now identify this family as the Shimura curve obtained above for $a=b=1/2$, by
computing its period matrices refining Bolza's method (see \cite[Section~11.7]{bl})
of comparing the action of automorphisms  on homology and differential forms.
\begin{Prop} \label{prop:perX_t}
The period matrices of the family $X_s$ are given by
\begin{equation}
\left(\begin{matrix}
\frac t2 - \frac12 & -\frac14 &  -\frac12\\
-\frac14 & \frac t2 & -\frac12 \\
-\frac12 & -\frac12 & -\frac12 + \frac i2\\
\end{matrix}
\right)
\end{equation}
for $t \in \HH$ or equivalently by
\begin{equation}
\left(\begin{matrix}
\frac t2 + \frac i4 - \frac14  & \frac i4 &\frac12 + \frac i2\\
\frac i4 & \frac t2 + \frac i4 + \frac14 & \frac12 + \frac i2 \\
\frac12 + \frac i2 &\frac12 + \frac I2 & i\\
\end{matrix}
\right)
\end{equation}
which agrees with the family in \eqref{gen3examples} for $u=\tfrac{1+i}{2}$ (i.e.~with \eqref{examples} for $a=b=1/2$) up
to the reparame\-teri\-zation $t\mapsto \frac t2 + \frac I4 + \frac14$.
\end{Prop}
The proof will be given in the rest of this section. We will not
discuss the uniformization of the Shimura curve, i.e.~we will not be interested in
the function relating $t$ in the matrices above to the parameter $s$ for the family.

For $s=1/2$  the special fiber $X_{1/2}$ is isomorphic to $w^2= z(z^6-1)$, and we claim that
its automorphism group is generated
by $T_1(z) = (1/z, i\, w /z^4)$ of order four
and $T(z,w) = (\zeta^2 z, \zeta w)$ of order twelve (where $\zeta$ be a primitive
$12$-th root of unity), with $T_1^2 = T^6$ the hyperelliptic involution.
In fact, $X_{1/2}$ is the unique hyperelliptic curve with this automorphism group
by \cite[Equation~3]{KuKo}.
The period matrix of this special fiber is calculated in \cite{schindler}, whose notation
we follow. He gives the period matrix as an element in $\HH_3$, but in order to
deform, we need the full set of periods with respect to a symplectic
basis on which we compute the action of the automorphisms.

As a basis of differential forms on $X_{1/2}$, we use $\{dz/w, zdz/w, z^2dz/w\}$,
which is an eigenbasis for the action of $T$ with eigenvalues
$\zeta$, $\zeta^3$, $\zeta^5$ respectively. We choose the symplectic
basis of homology given by the $\gamma_i$ and $\delta_i$, $i=1,2,3$
as in \cite[Section~2 and Figure~3]{schindler}. We then denote by
$\Pi \in \CC^{3\times 6}$ the corresponding full period matrix. The action of
$T$ and $T_1$ in this basis is given by the symplectic matrices

$$
M = \left(\begin{smallmatrix}
0&0&0&1&-1&0\\
0&0&0&0& 1&-1\\
0&0&0&0&0&1\\
-1&0&0& 0&1&0\\
-1&-1&0&1&0&0\\
-1&-1&-1&1&0&0\\
\end{smallmatrix}
\right)
\quad \text{and} \quad
M_1 = \left(\begin{smallmatrix}
0&0&0&1&0&0\\
0&0&0&0&0&-1\\
0&0&0&0&-1&1\\
-1&0&0&0&0&0\\
0&1&1&0&0&0\\
0&1&0&0&0&0\\
\end{smallmatrix}
\right),
$$
respectively, as given in \cite{schindler}, or as can be calculated directly
by tracing the action of $T_1$. If we denote by $L$ the diagonal
matrix with eigenvalues  $\zeta$, $\zeta^3$, $\zeta^5$ on the diagonal,
then the compatibility of the action on differential forms and on homology gives
$$ L\Pi = \Pi M$$
which determines $\Pi$ uniquely up to multiplying each row by a scalar. We
conclude that
$$
\Pi_{1/2} = \left(\begin{smallmatrix}
\zeta^2 + \zeta & \zeta^2 + 1 & 1 & -\zeta^2 + \zeta & -\zeta^3 & -\zeta\\
1/2 - \zeta^3/2 &0 &1 & -1/2-\zeta^3/2 & \zeta^3 & -\zeta^3\\
1-\zeta - \zeta^2 + \zeta^3 & 2-\zeta^2 & \,\,1\,\,& \zeta^4 + \zeta^5
& -\zeta^3 & -\zeta^5\\
\end{smallmatrix}
\right)
$$
The special fiber $X_{1/2}$ admits many maps to elliptic curves and to genus two
curves, and we need to single out the maps that deform to give global isogenies of
the family $X_s$ onto a product. To this end,
note that the global automorphism $T_0$ of the family restricts to $T^3$ on the
special fiber $X_{1/2}$. Hence we need
to compute the eigenspaces of an automorphism that deforms to an
automorphism of the general fiber and that moreover possesse a
non-trivial invariant subspace. One
such automorphism is $T_0 T^3$, for which we easily calculate its invariant
subspace and  $(-1)$-eigenspaces; thus writing everything in this basis, we see that
$$ B_D = \left(\begin{smallmatrix}
-\zeta^4/2 & 0 & 1/2 \\
0 & 1 & 0 \\
\zeta^4/2 & 0 & 1/2 \\
\end{smallmatrix}
\right)
\quad \text{and} \quad
B_H =  \left(\begin{smallmatrix}
0& 0& 1& 0& 0& 0\\
-1& 0& -1& 1& 0& 0\\
0& 0& 0& 0& 1& 0\\
-1& 0& 0& 0& 0& 1\\
-2& 0& 0& 0& 1& 0\\
-1& -2& -1& 1& 0& 0\\
\end{smallmatrix}
\right)
$$
diagonalize the action of $T_0T^3$ on homology and on differential forms
respectively, i.e.~$B_D^{-1}L_0L_1^3 B_D = {\rm diag}(1,1,-1)$ and
$$B_H M_0 M_1^3 B_H^{-1} = {\rm diag}(1,1,-1,1,1,-1).$$
(Note that $B_H$ has the eigenvectors in its rows, since the action on homology
on the period matrix is from the right.)
After this base change,
the period matrix becomes theblock matrix
$$ \Pi_B = B_D\, \Pi_{1/2}\, B_H^{-1} = \left(\begin{smallmatrix}
\zeta^2 + \zeta & \zeta^2 + 1 & 0 & 1 & -\zeta^2 + \zeta & 0\\
1/2-\zeta^4/2 & -1/2 - \zeta^4/2 & 0 &-\zeta^4 & \zeta^4 & 0\\
0&0& -\zeta/2 - \zeta^3/2 &0&0&\zeta^4/2 - 1/2\\
\end{smallmatrix}
\right)$$
The block of the two upper rows and columns $1,2,4,5$ of this matrix after a suitable
base change becomes
$$\Pi_2 = \left(\begin{smallmatrix}
-1-\frac53i & -1 + \frac13i & 2 & 0 \\
-1 + \frac13i & -1 + \frac23i  & 0 & 1 \\
\end{smallmatrix}
\right)
$$
and the left $2\times 2$-block of $\Pi_2$ is transformed by
$$ S = \left(\begin{smallmatrix}
0 & 0 & 1 & -1 \\
1 & 0 & 1 & 1 \\
0 & 1 & 1 & 1 \\
0 & 0 & 1 & 0 \\
\end{smallmatrix}
\right)
\in \Sp_4(\ZZ) \qquad \text{into} \quad Z_2 =
\left(\begin{smallmatrix}
\frac32i & \frac12 \\
\frac12 & \frac32i \\
\end{smallmatrix}
\right)
$$
To summarize, we write  $\Pi_{1/2} = (\Pi_{1/2}^\ell\,\,
\Pi_{1/2}^r)$ as two $3\times 3$ blocks, and consider the period matrix $Z_{1/2} = \bigl( \Pi_{1/2}^r
\bigr)^{-1}\, \Pi_{1/2}^\ell$. Let $S_3 \in \Sp_6(\ZZ)$ be the
image of $S$ under the natural upper left block inclusion
$\Sp_4(\ZZ) \to \Sp_6(\ZZ)$. We conclude that the isogeny
from the Jacobian of $X_{1/2}$ to the product abelian variety with period matrix
$\left( \begin{smallmatrix} Z_2 & 0 \\ 0 & i \\ \end{smallmatrix} \right)$
is given by $S_3 \,D \,(B_H^{-1})^T$, where $D = {\rm diag}(1,1,1,2,1,1)$ (since the action
on the full period matrix by a base change on homology corresponds
to the fractional action on the period matrix in $\HH_3$ by the transpose
matrix).

On the other hand we now consider the family $Y_s$ of quotients
of $X_s$ by the involution  $T_0 T^3$. It is given by the
affine equation
$$Y_s =\lbrace y^4 = x^2(x-1)((s+1)x+(1-s))^2\rbrace$$
with the quotient map given by
$$(z,w) \mapsto (x,y) =  \left( \Bigl(\frac{z^2 -1}{z^2+1}\Bigr)^2, \frac{w}{2(z^2+1)^2}\right).$$
This is one of the families of genus two curves, whose automorphism
group is sufficiently large to directly determine the
period matrix. In fact, besides the cyclic group of order four given
by $(x,y) \mapsto (x,iy)$ the involution that acts on $\PP^1_x$ by
the permutation $(0\, \tfrac{1-s}{s+1})(1\, \infty)$ on the branch points
lifts to an automorphism of $Y_s$. Consequently,
the reduced automorphism group of $Y_s$
is the Klein four group, and the period matrices, given
in \cite[Type II in the table p.~340]{bl} are
$$Z_S =  \left(\begin{smallmatrix}
S & \frac12 \\
\frac12 & S \\
\end{smallmatrix}
\right).
$$
where we note that $S$ in this parametrization is not the same as the parameter $s$ that we are using --- in particular, the special fiber $s=1/2$ corresponds to
$S = \tfrac32i$.

Since the set of isogenies between any pair of abelian
varieties is a discrete datum, they persist under
small deformation. We conclude that the
period matrices of the family $X_s$ are given by
$B_H^T D S_3^{-1} \left(\begin{smallmatrix}
Z_S & 0 \\
0 & i \\
\end{smallmatrix}
\right).
$
Composition with the symplectic matrices
$$ C_1 =  \left(\begin{smallmatrix}
1& 0& 0& 0& 0& 0 \\
0& 0& 1& 0& 0& 0 \\
1& 0& 0& 0& 1& 0 \\
-1& 1& 0& 1& 0& 0\\
0 & 0& 0& 0& 0& 1\\
0 &-1& 0& 0& 0& 0\\
\end{smallmatrix}
\right)
\quad \text{or} \quad C_2 =  \left(\begin{smallmatrix}
1 &0 &0 &0 &0 &0\\
0 &1 &0 &0 &0 &0\\
0 &0 &-1 &0 &0 &-1\\
0 &0 &0 &1 &0 &0\\
0 &0 &0 &0 &1 &0\\
0 &0 &1 &0 &0 &0\\
\end{smallmatrix}
\right) C_1
$$
gives
$$ C_1 B_H^{-1} S_3^{-1} \left(\begin{smallmatrix}
Z_S & 0 \\
0 & i \\
\end{smallmatrix}
\right)
=
\left(\begin{smallmatrix}
1 &0 &0 &-1& -1& -1 \\
0 &1 &0 &-1& 0 &-1 \\
0 &0 &1 &-1& -1& -1\\
0 &0 &0 &2 &0 &0 \\
0 &0 &0 &0 &2 &0 \\
0 &0 &0 &0 &0 &2 \\
\end{smallmatrix}
\right)
\left(\begin{smallmatrix}
Z_S & 0 \\
0 & i \\
\end{smallmatrix}
\right)
=
\left(\begin{matrix}
\frac t2 - \frac12 & -\frac14 &  -\frac12\\
-\frac14 & \frac t2 & -\frac12 \\
-\frac12 & -\frac12 & -\frac12 + \frac I2\\
\end{matrix}
\right).
$$
which gives a period matrix in the form required by
Lemma~\ref{le:V2ndZmatrixform}. The base change $ C_2 B_H^{-1} S_3^{-1}$ gives the
second version of the period matrix
and completes the proof of Proposition~\ref{prop:perX_t}.


\end{document}